\documentclass[10pt]{amsart}
\usepackage{amssymb, amsmath, amsthm}
\usepackage{cancel}
\usepackage[all]{xy}
\usepackage{empheq}
\usepackage{extpfeil}
\usepackage{tikz}
\usetikzlibrary{arrows,decorations.markings,fit,matrix,positioning,calc,shapes.symbols,shapes,intersections}
\usepackage{hyperref}

\newtheorem{theorem}{Theorem}[section]
\newtheorem{lemma}[theorem]{Lemma}
\newtheorem{proposition}[theorem]{Proposition}
\newtheorem{corollary}[theorem]{Corollary}

\theoremstyle{definition}
\newtheorem{remark}[theorem]{Remark}
\newtheorem{definition}[theorem]{Definition}

\def\beq{\begin{eqnarray*}}
\def\eeq{\end{eqnarray*}}
\def\Q{\mathbb{Q}}

\def\R{\mathbb{R}}
\def\Z{\mathbb{Z}}
\def\In{\mathbb{I}}

\def\La{\widetilde{L}}

\def\incl{\hookrightarrow}
\def\to{\rightarrow}

\def\eps{\varepsilon}

\def\x{\times}
\def\d{\partial}
\def\phi{\varphi}
\def\ttot{\widetilde{\mathrm{Tot}}}

\def\Emb{\mathrm{Emb}}

\def\P{\mathcal{P}}

\def\Map{\mathrm{Map}}

\def\fib{\mathrm{fib}}
\def\tfib{\mathrm{tfib}}

\def\holim{\mathrm{holim}}
\def\C{\mathcal{C}}

\def\la{\langle}
\def\ra{\rangle}

\def\AM{\widetilde{AM}}

\def\tC{\widetilde{C}}

\def\co{\colon\thinspace}

\def\G{\mathcal{G}}

    \title{Embedding calculus knot invariants  are of finite type}
    
    \author{Ryan Budney}
    \email{ryan.budney@gmail.com}
        \address{Department of Mathematics, University of Victoria, Victoria, BC, Canada}

    \author{James Conant}
    \email{jim.conant@gmail.com}
    \address{Department of Mathematics, University of Tennessee, Knoxville, TN, USA}
    
    \author{Robin Koytcheff}
    \email{koytcheff@math.umass.edu}
        \address{Department of Mathematics and Statistics, University of Massachusetts-Amherst, Amherst, MA}

    \author{Dev Sinha}
    \email{dps@uoregon.edu}
        \address{Department of Mathematics, University of Oregon, Eugene, OR, USA}

    \thanks{}

\begin{document}

\maketitle

\begin{abstract}
We show that the map on components from the space of classical long knots to the $n$th stage of its Goodwillie--Weiss embedding calculus tower is a map of monoids whose target is an abelian group, and which is invariant under clasper surgery.  We deduce that this map on components is a finite type-$(n-1)$ knot invariant.  We  compute the $E^2$-page in total degree zero for the spectral sequence converging to the components of this tower, identifying it as  $\Z$-modules of primitive chord diagrams, providing evidence for the conjecture that the tower is a universal finite-type invariant over the integers.  Key to these results is the development of a group structure on the tower compatible with connect-sum of knots, 
which in contrast with the corresponding results for the (weaker) homology tower requires novel techniques involving operad actions, evaluation maps,  and cosimplicial and subcubical diagrams.  
\end{abstract}

\section{Introduction} 

In this paper we connect three current threads in studying knots and the moduli space of all knots:  
 the Goodwillie--Weiss embedding calculus, Budney's operad actions, and Vassiliev's theory of finite-type invariants.  
 This work also connects two fundamental results on commutativity which are over fifty years old.
In 1949 H. Schubert \cite{Schubert} established  and applied the fact that connect-sum of knots is commutative. 
 In 1947 Steenrod \cite{Steenrod} gave explicit formulae exhibiting commutativity of cup product, in the course of 
 defining the cohomology operations which bear his name.  We establish and use compatibility of these classical 
 results as we develop a group structure on the components of the Goodwillie--Weiss tower.

One of our main results is that the Goodwillie--Weiss tower for  knots yields additive invariants of finite type.  
Such a result is modest compared to the  conjecture made in \cite{BCSS}  that the tower gives all such invariants. 
We provide  evidence for this conjecture through calculation of  the $E^2$-term of the spectral sequence for the 
components of this tower.
The current results are meaningful first steps, and establishing them requires  new  
combinations of tools in the areas of operad actions, clasper surgery, cosimplicial and cubical diagrams, and evaluation maps.

A large part of the work is ``getting off the ground,'' showing that the Goodwillie--Weiss tower yields 
abelian group-valued invariants.  Because the knot invariants are defined as an induced map on $\pi_0$, 
such invariants are a priori only set valued.    Once we establish the group structure, we use the 
Habiro surgery criterion for finite type to establish the main result that the $n$th stage in the tower defines
a type-$(n-1)$ invariant.

Our work is a natural sequel to two previous papers.  The first is \cite{BCSS}, in which we establish 
that the third stage of the tower is a universal type-two invariant.  Here the group structure and the 
finite-type result were straightforward from ad hoc arguments.  The more interesting aspect of this work 
was the development of geometry to explicitly distinguish components in the tower, yielding a new  
interpretation of the type two knot invariant by counting collinearities of points.  This geometric approach was continued in \cite{Flowers}.

The second predecessor  is Volic's thesis \cite{VolicFT}, which establishes that the 
Goodwillie--Weiss tower for the (space level) rational homology of knots is a universal rational finite-type invariant.  
The map to the homology tower factors through the tower we consider (the ``homotopy tower''),
which thus encodes such a universal rational invariant as well.
The  question of whether the tower gives a universal invariant
over the integers is of considerable interest since Vassiliev invariants with values in arbitrary $\mathbb Z$-modules
are not well understood. 
A universal invariant is known to exist over the integers (see for example Habiro's $\psi_n$ map, used Theorem~6.1 of this  paper), 
but it is not computable nor does it directly relate to the combinatorics of weight systems.
Identification of the tower as a universal invariant would likely resolve questions about the
combinatorics and geometry of finite-type invariants, and in particular the open question of whether weight systems ``integrate.''

Our techniques are  distinct from those of Volic in three ways.   
In our result the type-$n$ invariants are potentially realized at the $(n+1)$st stage, 
while in the rational homology tower they are realized precisely  at stage $2n$.   
Volic uses the Vassiliev derivative approach to establish his finite-type result, while we use Habiro's clasper surgery.
Finally, Volic's result proceeds by extending the theory of Bott-Taubes integration to the homology tower, 
while our approach through the homotopy tower 
invites new techniques from  geometric and algebraic topology.

In the study of multiplicative structures on  the Goodwillie--Weiss tower, there are further predecessors to our
work, namely results of Sinha  \cite{SinhaOperads} as well as of Turchin \cite{VictorDelooping} and 
Dwyer-Hess \cite{DwyerHess} on deloopings of  multiplicative structures on the tower.  Indeed, by  Turchin's result one
can deduce that the components of stages of the tower for classical framed knots is an abelian group. 
(Dwyer-Hess's result applies to the limit.  Sinha's result only implies an abelian monoid structure,
and also only applies to the limit.)  
But our application requires compatibility with 
connect sum of knots, which we could not establish using these models.  Thus there are now four different monoid structures on the components 
of the tower for framed classical knots, three of which are group structures, which are all conjecturally equivalent. 

We extend the abelian group structures to the spectral sequence level, 
which is crucial for analysis since studying components of an arbitrary tower of spaces can potentially lead to unending ad-hoc calculations.  
To do so we employ $C_1$-operad actions, which necessitates their development throughout our work
despite the fact that our main  theorems are at the level of components.
With group structure in hand, we immediately establish convergence at each finite stage and then use results of the second author \cite{Conant}
to identify the $E^2$ term in degree zero as the $\Z$-modules of primitive chord diagrams.  

Based on this, we conjecture that the map from the knot space to the tower sends linear combinations of knots given by 
resolving a singular knot to the corresponding elements of $E^2$.  We also conjecture that the spectral sequence collapses, 
and together these two conjectures would imply that weight systems over the integers all integrate to finite-type invariants.  
Unlike Vassiliev's original spectral sequence, this spectral sequence does not involve a subtle limit process (see \cite{Giusti}) 
but instead is simply the spectral sequence of a tower of fibrations.  It thus is more amenable to  tools from algebraic topology.

\bigskip

The paper is organized as follows.  Section \hyperref[Background]{\ref*{Background}} gives needed general background on compactified configuration spaces and 
on cubical and cosimplicial diagrams.   Section \ref{TheModels} recalls the resulting mapping space and cosimplicial models we prefer to use, 
interchangeably, for the $n$-th stage of the Taylor Tower for $\Emb^{fr}(\R, \R^3)$.  Some readers may  prefer to only look back at these sections 
as needed, in particular for the ``infinitesimal transformation'' map between the spatial and operadic mapping space models. 

In Section \ref{LittleIntervals}, we construct homotopy-commutative multiplications on the $n$-th stage of the Taylor tower for $\Emb^{fr}(\R, \R^3)$.    This shows that $\pi_0$ of each stage of the Taylor tower is an abelian monoid. Section \ref{PropertiesOfProjMaps} contains two main results, the first being that the projections in the tower are surjective on $\pi_0$.  This is then used to show that $\pi_0$ of each stage is a group.


In Section \ref{FiniteType}, we show that $\pi_0$ of the map from the space of knots to its $(n+1)$st Goodwillie--Weiss approximation is invariant under
clasper surgery and thus of type $n$.  In Section \ref{SpectralSequence}, we show that  the homotopy spectral sequence for the tower is a spectral sequence of abelian groups (in particular in total degrees zero and one), and we identify the $E^2$ term.

\subsection{Acknowledgments}
The authors thank the referee for a careful reading of the paper and useful comments.  The first author acknowledges support from NSERC and PIMS.  The third author was supported partially by NSF grant DMS-1004610 and partially by a PIMS postdoctoral fellowship.  He thanks Tom Goodwillie for useful conversations.

\tableofcontents

\section{Compactification of configuration spaces}\label{Background}
\label{SimplicialCptfn}
%

\subsection{Basic definition}
We briefly review the simplicial compactification of configuration space, defined in \cite{SinhaTop, SinhaCptn}.
For any manifold $M$ -- not necessarily compact, and possibly with boundary -- let $C_n(M)$ denote the configuration space of $n$ points in $M$.  It is the space of distinct ordered $n$-tuples of points in $M$.  

\begin{definition}
For an $M$ embedded in some Euclidean space $\R^d$, let $C_n\la M \ra$ denote the closure of the image of 
\begin{align*}
C_n(M) & \incl M^n \x (S^{d-1})^{n \choose 2}  \\
(x_1,...,x_n) & \mapsto (x_1,...,x_n), \left( \frac{x_i - x_j}{|x_i - x_j|} \right)_{i<j}
\end{align*}
\end{definition}

As shown in \cite{SinhaCptn}, this is independent of embedding and is a quotient of $C_n[M]$, 
the Fulton--MacPherson compactification of $C_n(M)$.  
If $M$ is noncompact, then $C_n\la M \ra$ as defined above as well as the Fulton-MacPherson compactification are not compact but are more accurately described as completions.  Informally, we refer to  the projection to $M^n$ as the ``spatial'' information, while information which distinguishes points with
the same projection to $M$ is ``infinitesimal.''
If $M$ is one-dimensional and connected, then $C_n\la M\ra$ is in general not connected.  For such $M$, by abuse we use $C_n\la M\ra$ to denote the connected component where the $n$ points are in (cyclic) order.


\subsection{Framings and tangent data}
We need framed configurations. For any manifold $M$, define $C_n^{fr}\la M\ra$ as the pullback
\[
\xymatrix{
C_n^{fr}\la M \ra \ar[r] \ar[d] & (Fr M)^n \ar[d] \\
C_n \la M \ra \ar[r] & M^n,
}
\]
where $Fr M \to M$ is the unit frame bundle of the tangent bundle of $M$.  If $M$ is parallelizable and $d$-dimensional, then $C_n^{fr}\la M \ra$ is homeomorphic to $C_n \la M\ra \x O(d)^n$.  

Let $C_n'\la M \ra$ be defined similarly, with the unit tangent bundle $STM$ in place of the unit frame bundle.



\subsection{Distinguished boundary points}
Suppose $M$ has two distinguished points $y_0,y_1$ in its boundary.  Then, as in \cite{BCSS} and \cite{SinhaTop}, let $C_n\la M, \d\ra$ denote the subspace of $C_{n+2}\la M\ra$ where the first and last points are located at $y_0$ and $y_1$, by abuse omitting dependence on these points from notation.  

Further, if there are distinguished tangent vectors $v_0 \in TM|y_0$ and $v_1 \in TM|y_1$, let $C_n'\la M, \d\ra$ be the subspace of $C_{n+2}'\la M \ra$ where $(p_1, v_1)$ and $(p_{n+2}, v_{n+2})$ are $(y_0, v_0)$ and $(y_1,v_1)$ respectively. 
By fixing framings at $y_0$ and $y_1$, define  $C_n^{fr}\la M, \d\ra$ similarly.

Define $C_n\la \In, \d\ra$ by taking the two endpoints to be the distinguished points.  
The fact that $C_n\la \In,\d\ra$ is 
the $n$-simplex is the main rationale for calling this compatification ``simplicial.''
Define $C_n' \la \In^d\ra$ by taking $\{y_0, y_1\}$ to be $\d \In \x \{(0,...,0)\}$ and $v_0 = v_1 = (1,0,...,0)$ (so our knots will ``proceed from left to right''), 
and similarly define $C_n^{fr} \la \In^d, \d\ra$ by using the identity element in $O(d)$ for framings at those boundary points 
(using the standard parallelization of $\In^d \subset \R^d$).

\subsection{Quotients by translation and scaling, and insertion maps}
\label{InsertionMaps}
There are maps between products of  $C_n \la \R^d\ra$ and  $C_n \la \In^d, \d \ra$ defined by ``inserting an infinitesimal configuration into a point of another configuration.''  
Let $\widetilde{C_n}(\R^d):=C_n(\R^d) / (\R^d \rtimes \R_+)$ be the quotient of configuration space by translations and positive scalings of all $n$ points.  
\begin{definition} \label{ConfigDef}
Define
$\widetilde{C_n}\la \R^d \ra$ as the closure of the image of the map 
\begin{align*}
\widetilde{C}_n(\R^d) &\overset{e}{\to} (S^{d-1})^{{n \choose 2}} \\
(x_1,...,x_n) &\mapsto \prod_{i < j} \frac{x_i - x_j}{|x_i - x_j|},
\end{align*}
which is injective except on collinear configurations.  

We let $v_{ij}$ denote the projection of $\widetilde{C_n}\la \R^d \ra$ to the $i,j$th factor of $S^{d-1}$.
\end{definition}

The similarly defined  $\widetilde{C}_n \la \In^d \ra$ is homeomorphic to $\tC_n \la \R^d \ra$, so both $C_n \la \R^d \ra$ and $C_n \la \In^d \ra$ naturally surject onto $\widetilde{C_n}\la \R^d \ra$.

We proceed directly to the framed setting in defining insertion maps.
Let $\widetilde{C}_n^{fr}\la \R^d \ra := \widetilde{C_n}\la \R^d\ra \x (O(d))^n$, the framed version of the ``infinitesimal configuration space."  

For every $m,n$, and $i \in \{1,...,n\}$, we define a map $\circ_i$ which,  informally, inserts a configuration of $m$ points 
(with framings) into the $i^\mathrm{th}$ point of a configuration of $n$ points (with framings).  
In the resulting configuration of $m+n-1$ points, the $m$ points form an 
``infinitesimal configuration."  Precisely, in coordinates we define 
\[
\circ_i: {C}^{fr}_n\la \In^d, \d \ra \x {C}^{fr}_m\la \In^d \ra \to {C}^{fr}_{m+n-1}\la \In^d, \d \ra
\]
as follows.  First, suppressing the dependence on $i$, we let
\begin{align*}
\widehat{j} = 
\left\{
\begin{array}{ll}
j & \mbox{if  $j\leq i$} \\
i & \mbox{if $i \leq j \leq i+m-1$} \\
{j+1-m} & \mbox{if $j > i+m-1 $.} 
\end{array}
\right. \\
\end{align*}
Now define $\circ_i$ as sending 
\[
((x_j)_{j=1}^n, (u_{jk})_{j<k}, (\alpha_j)_{j=1}^n,) \x ((y_j)_{j=1}^m, (v_{jk})_{j<k}, (\beta_j)_{j=1}^m) \mapsto ((z_j)_{j=1}^{m+n-1}, (w_{jk})_{j<k}, (\gamma_j)_{j=1}^{m+n-1}),
\] 
where $z_j = x_{\hat{j}}$, 

\begin{align*}
w_{jk} = 
\left\{
\begin{array}{ll}
\alpha_i v_{(j-i+1)(k-i+1)} & \mbox{if $i \leq j \text{ and } k \leq i+m-1$} \\
u_{\hat{j}\hat{k}} & \mbox{otherwise},  \\
\end{array}
\right. \\
\end{align*}
and 
\begin{align*}
\gamma_j =
\left\{
\begin{array}{ll}
\alpha_i \beta_{j-i+1} & \mbox{if $i \leq j \leq i+m-1$} \\
\alpha_{\hat{j}} &  \mbox{otherwise}. 
\end{array}
\right. \\
\end{align*}

One must check that such maps send  subspaces of $(\R^d)^n \x (S^{d-1})^{n \choose 2}$ 
to each other appropriately.  Here we only cite a similar check, namely
\cite[Proposition 6.6]{SinhaCptn}, using the description  of the $C_n\la \R^d \ra$ as a subspace of 
 $(\R^d)^n \x (S^{d-1})^{n \choose 2}$ given in \cite[Theorem 5.14]{SinhaCptn}.
The results are not dependent on the $y_j$ coordinates of ${C}^{fr}_m\la \In^d \ra$, which means that these
insertion maps factor through the projection to $\widetilde{C}^{fr}_m\la \R^d \ra$ on that factor.

\section{The models} 
\label{TheModels}

We study the space of framed knots $\Emb^{fr}(\R, \R^3)$.  A framed knot is a 
 smooth embedding of $\R$ into $\R^3$ together with a smooth map $\R \to O(3)$
whose first (column) vector is  the unit derivative map.  
Embeddings take $\In =[-1,1] \subset \R$ into $\In^3 \subset \R^3$, and on $\R \setminus \In$ are 
standard, given by $t \mapsto (t,0,0)$.  The framing is required to be constant at the identity on $\R \setminus \In$.

We primarily use a mapping space model
for the $n$-th stage of the Goodwillie--Weiss tower for the 
space of framed knots $\Emb^{fr}(\R, \R^d)$.  But for both spectral sequence calculations and for clarity through comparison,
we use a cosimplicial model as well.  For each of these, there are a few variants depending on the choice of compatifications
of  configuration spaces $C_n(M)$.  The  Fulton--MacPherson (Axelrod--Singer) compactification $C_n[M]$
 is a smooth manifold with corners.  
In this paper, we use instead the simplicial compactification $C_n\la M \ra$, developed in the previous section.  
It is not a manifold with corners, but it has the advantage that one component of
$C_n \la \In \ra$ is the $n$-simplex.   This is needed to define a cosimplicial model, and the 
corresponding mapping space model is defined in terms of the face poset of the simplex rather
than that of the associahedron.

\subsection{Cosimplicial models}

A cosimplicial space is a functor to $\mathcal{T}op$ from the category $\Delta$ with one object for 
each ordered set $[n]=(0, \ldots, n)$ and morphisms given by order-preserving maps.  
We rely on the standard set of generating morphisms, denoting coface maps by $d^i$ and codegeneracy maps by $s^i$.
We let $\Delta_n$ be the full sub-category of $\Delta$ containing the first $n+1$ objects.

Goodwillie-Weiss embedding calculus leads to forming a cosimplicial space from 
the spaces $C_\bullet^{fr} \la \In^d, \d \ra$, 
by using the first vector in the framing as the ``doubling direction.''
This was first done in the unframed setting in \cite[Corollary 4.22]{SinhaTop}.
In more detail, we have the following.

\begin{definition}\label{CosimplicialModelSpatial}
 The spatial cosimplicial model  ${C}^{fr}_\bullet \la \In^d, \d \ra$ has $n^\mathrm{th}$ entry ${C}^{fr}_n \la \In^d, \d \ra$.  
 
 The codegeneracy $s_i: {C}_n^{fr} \la \In^d, \d \ra \to {C}_{n-1}^{fr} \la \In^d, \d \ra$ is the extension to 
 compactifications  of the projection which forgets the $i^\mathrm{th}$ point.  
 
 The coface $d^i$ is given by $d^i(x)= x\circ_i \mu$, which ``doubles'' the $i^\mathrm{th}$ 
 point by inserting into its position the infinitesimal two-point configuration $\mu$ rotated by the $i$-th framing.  
 \end{definition}
 
 So here $d^0$ and $d^{n+1}$ ``double'' the ``extra'' points which are located in the middle of the left and right faces of $\In^d$.
 
 \begin{definition}\label{CosimplicialModelOperadic}
 The operadic cosimplicial model $\widetilde{C}^{fr}_\bullet \la \R^d\ra$ has $n^\mathrm{th}$ entry $\widetilde{C}^{fr}_n \la \R^d\ra$.  
 
 The codegeneracy $s_i: \widetilde{C}_n^{fr}  \la \R^d \ra \to \widetilde{C}_{n-1}^{fr} \la \R^d \ra$ is the extension to compactifications
 modulo translation and 
 scaling of the projection which forgets the $i^\mathrm{th}$ point.  
 
 For $1\leq i \leq n$, the coface $d^i$ is given by $d^i(x)=x\circ_i \mu$, which ``doubles'' the $i^\mathrm{th}$ 
 point by inserting into its position the infinitesimal two-point configuration $\mu$ rotated by the $i$-th framing.  
 The coface $d^0(x)=\mu \circ_2 x$, while the coface  $d^{n+1}(x) = \mu \circ_1 x$.
 \end{definition}

Thus in the operadic model,  the first and last coface maps insert configurations into two-point configuration, 
which has the effect of ``adding a point at infinity.''  Along with the obvious passing to quotients by translation and 
scaling, this is the difference
between the spatial and operadic models.  We discuss the relative advantages of these models as well
as why we need both when we discuss their associated mapping space models below.  The operadic model
is called such because it fits into a framework by which operads with multiplication produce cosimplicial
objects, following 
Gerstenhaber--Voronov and McClure--Smith \cite{MS}.

 The fourth author showed \cite{SinhaTop, SinhaOperads} 
 that the homotopy-invariant totalizations of similar cosimplicial spaces for unframed knots, which we will use and 
 denote by ${C}'_\bullet \la \In^d, \d \ra$ and $\widetilde{C}'_\bullet \la \R^d\ra$, give models for the Goodwillie--Weiss tower.
 The framed version of this construction was  studied  by Salvatore in   \cite[Section 3]{Salvatore}.

 By needing to use a homotopy-invariant totalization to model the towers, some control of geometry and combinatorics is lost.  
 A standard approach to cosimplicial spaces through (sub)cubical diagrams, reviewed in
 the next section, retains 
 both combinatorics and geometry by sacrificing some symmetry.  In particular the resulting spatial model is 
 compatible with the evaluation map, also known as a Gauss map, from the knot space.

\subsection{Mapping space models}
\label{TheMappingSpaceModel}

Our mapping space models are defined as homotopy limits of subcubical diagrams of compactified configuration spaces.
A subcubical diagram is a functor from
$\P_\nu[n]$, the poset  of nonempty subsets of $[n]$, which is the face poset of the $n$-simplex.
Because the cosimplicial and subcubical diagram categories both involve ordered sets, there is an immediate relationship
between them.  In general $\Delta$ admits a canonical functor from any category defined through finite ordered sets, as follows.

\begin{definition}\label{GFunctor}
Let $C$ be a category whose objects are given by 
ordered finite sets and whose morphisms are subsets of the order-preserving maps between those sets.
Define $\mathcal{G}_C : C \to \Delta$ to be the functor which sends an ordered finite set $S$ to $[\# S -1]$ and which sends an order-preserving
map $S \to T$ to the composite 
$$[\#S -1] \cong S \to T \cong [\#T -1],$$
where the isomorphisms are order-preserving.

For  $C = \P_\nu[n]$, we abbreviate $\mathcal{G}_{\P_\nu[n]}$ to $\mathcal{G}_n: \P_\nu[n] \to {\Delta_n}$.
\end{definition}

The functor $\mathcal{G}_n$ was used \cite{SinhaTop} and \cite{MunsonVolicCubes} to use cubical diagrams to model cosimplicial spaces.
The resulting diagrams use all of the coface maps ``multiple times'', while the codegeneracy maps
are ignored.


\begin{definition}\label{AMmodels}

The spatial mapping space model $AM_n^{fr}$ is the homotopy limit of the composite functor 
$C_\bullet^{fr} \la \In^3, \d \ra \circ \mathcal{G}_n :\P_\nu[n] \to \mathcal{T}op$.  

The operadic mapping space model $\AM_n^{fr}$ is the homotopy limit of the composite functor 
$\widetilde{C}_\bullet^{fr} \la \R^3 \ra \circ \mathcal{G}_n :\P_\nu[n] \to \mathcal{T}op$.

\end{definition}

We are henceforth focussing on classical knots, in three dimensions, so we are suppressing the ambient dimension from notation.
The mapping space model $AM_n$ of \cite{BCSS} is defined similarly to $AM_n^{fr}$, but pulled back from ${C}'_\bullet \la \In^3, \d \ra$ 
and so with tangent vectors instead of frames. 

Using the definition of homotopy limit through nerves of under-categories, 
the homotopy limit of a subcubical digram is given by a collection
of maps from simplices.  Since the structure  maps $d^i$ are injective, 
an element $\phi$ of $AM^{fr}_n$ is determined by a map $\Delta^n \to {C}^{fr}_n \la \In^3, \d \ra$, 
so $AM^{fr}_n$ is  a subsspace of $\Map(\Delta^n, {C}^{fr}_n\la \In^3, \d \ra)$.  
Because the faces of the simplex map to configurations which are degenerate in an ``aligned'' manner we sometimes refer to this as the 
subspace of aligned maps.  Explicitly, if some (consecutive) collection of $t_i$ in $\vec{t}=(t_1,...,t_n) \in \Delta^n$ are equal, 
then the corresponding points in the configuration $\phi(\vec{t})$ have ``collided'' in $\In^3$, their framings ($\alpha_i \in O(3)$) are all equal, and 
the first column of $\alpha_i$ is the direction of collision of these points.

 To interrelate cosimplicial and mapping space 
 models, a main technical result is Theorem 6.7 in \cite{SinhaTop}, which establishes that  $\mathcal{G}_n$ is left cofinal.
So if  $X^\bullet$ is a cosimplicial space then the homotopy limit of the subcubical diagram $X^\bullet \circ \mathcal{G}_n$ is 
equivalent to the homotopy limit of the restriction of $X^\bullet$ to $\Delta_n$.   By work of Bousfield and Kan \cite{BK}, this is homotopy equivalent to 
$\ttot^n X^\bullet$, the $n$th  stage in the homotopy invariant totalization tower.   In \cite{SinhaTop} this is used to establish the validity
of the cosimplicial models, building from that of the mapping space models.


\subsubsection{Evaluation maps}

Our main results make use of evaluation maps, which naturally connect with the spatial mapping space model.
By functoriality for embeddings of compactified configuration spaces (Corollary 4.8 of \cite{SinhaCptn}), we have that an embedding
$f\co \In \to \In^3$
will extend to a map from $\Delta^n =C_n \la \In, \partial \ra$ to $C_n \la \In^3, \d \ra$.   For a framed embedding
with framing $\alpha \in \Map(\In, O(3))$ to go along with the embedding $f$, this map is given as follows.

\begin{definition}
Define
$ev_n: \Emb^{fr}(\R, \R^3)  \to  AM^{fr}_n \subset \Map(\Delta^n, C_n\la \In^3, \d \ra \x O(3)^n)$ as sending an embedding $f$ and framing $\alpha$
to the map which sends
$$ ((-1\leq t_1 \leq ... \leq t_n \leq 1) \mapsto \left( (f(t_1),...,f(t_n)), \alpha(t_1),..., \alpha(t_n) \right).$$

\end{definition}

Intuitively, $ev_n$ samples the knot and its framing at $n$ points in the domain. 


One of the main results of \cite{SinhaTop}, 
namely Theorem 5.4, immediately extends to the framed setting to establish 
that $ev_n$  agrees in the homotopy category with the canonical map from the embedding space 
$\Emb^{fr}(\R, \R^d)$ to the $n$-th stage of  the Taylor tower $T_n \Emb$.   

In dimensions $d>3$ it is known that the connectivity of $ev_n$ increases with $n$, and hence the tower 
converges to the embedding space.  
Thus, the homotopy theory of evaluation maps captures the nature of knots in these dimensions.  
For $d=3$, it is not known if $ev_n$ is even 0-connected (that is surjective) much less injective in the limit.
We show below that on components $ev_n$  defines abelian-group-valued invariants of finite type $(n-1)$.

\subsubsection{Translation between mapping space models}\label{OneIntervalAction}

In addition to the quotienting by translation and scaling, $AM_n$ and $\AM_n$ differ by the first and last coface maps, which 
in the former case ``add a point to the configuration on the left or right face of $\In^3$'' while in the latter case such points are added
``at infinity to the left or right.''  This means that the obvious
quotient maps do not define a map between these models.  We require a map between them, as the spatial model is the 
target of the evaluation map as we just saw and the operadic model is more convenient for defining multiplications below.

Observe that the evaluation of a long knot can be extended to include times greater than $1$ or 
less than $-1$, in which case the corresponding configuration points are standard along the $x$-axis.
We achieve something similar for arbitrary elements of $AM_n$ to obtain maps 
$\tilde{\varphi} : C_n \la \R \ra  \to C_n \la \R^3, \partial \ra$.  

\begin{definition}\label{Extension}
For $t_1 \leq t_2 \leq \cdots \leq t_n$ with all $t_i \in \R$, we let $\hat{t_i} = \begin{cases} -1 & t_i \leq -1 \\ t_i   & -1 \leq t_i \leq 1.\\  1 & t_i \geq 1 \end{cases}$ 

Then we define $\tilde{\varphi}(t_1 \leq t_2 \leq \cdots \leq t_n)$ to be the configuration which is the union of the following:
\begin{enumerate}
\item One point at $(t_i, 0, 0)$ for teach $t_i \leq -1$ or $\geq 1$.
\item The configuration obtained by taking $\varphi(\hat{t_1} \leq \cdots \leq \hat{t_n})$ and applying the projection which forgets points $x_i$ for which
 $t_i \leq -1$ or $\geq 1$.
 \end{enumerate}
 
 Moreover after we compose with the quotient $C_n \la \R^3, \partial \ra \to \widetilde{C_n} \la \R^3, \partial \ra$, 
 $\tilde{\varphi}$ extends to a map from $C_n \la [-\infty, \infty] \ra$ to $\widetilde{C_n} \la \R^3, \partial \ra$ since the limit as some $t_i$ goes to $-\infty$ or $\infty$ will have all vectors $v_{ij}$ approaching $(1,0,0)$ or $(-1,0,0)$, as all points in the resulting configurations not in $\In^3$ lie 
 along the $x$-axis.
 \end{definition}
 
 Let $h$ be an order-preserving homeomorphism between $[-1, 1]$ and $[-\infty, \infty]$.
 By our usual abuse of notation, let $h$ also denote the induced map on collections of points in these.
 
 \begin{definition}\label{Iota}
 Let $\iota : AM_n \to \AM_n$ be the map which sends $\varphi$ to the composite 
 $$\Delta^n \overset{h}{\to} C_n \la [-\infty, \infty] \ra \overset{\tilde{\varphi}}{\to}  \widetilde{C_n} \la \R^3, \partial \ra.$$
 \end{definition}

\begin{proposition}\label{IotaHomotopyEquivalence}
$\iota$ is a homotopy equivalence.
\end{proposition}

\begin{proof}
To compare $ AM_n$ and $\AM_n$ it is simplest to map them both to a third space which is easily seen to be equivalent.  
Let $\AM^{hs}_n$, where $hs$ stands for hemispherical, be the subspace of $\Map(\Delta^n, \widetilde{C_n} \la \In^3 \ra)$ defined by the same conditions
as for $\AM_n$ for all faces except the $t_1 = -1$ and $t_n = 1$ faces.  On those faces, instead of the vectors $v_{1j}$ and $v_{kn}$ in being $(1,0,0)$ as
in the definition of $\AM_n$,
for $\AM^{hs}_n$ those vectors are simply required to lie in the hemisphere with nonnegative $x$-coordinate. (We could also define this as a homotopy limit.)

Consider the  diagram
\[
\xymatrix{
 AM_n  \ar[d]_-{q}^-\simeq \ar[dr]^-{\iota} & \\
  \AM^{hs}_n  &  \AM_n, \ar[l]^-{i}_-\simeq
}
\]
where $q$ is the map induced by composition of $\Delta^n \to C_n \la \In^3, \partial \ra$ with the quotient map to $\widetilde{ C_n} \la \In^3 \ra$
and $i$ is the inclusion.  Both of these are homotopy equivalences.  The diagram commutes up to homotopy,  interpolating  between $q$ and $i \circ \iota$
by first extending $\tilde{\varphi}$ to $[-x, x]$.  Then use a continuous family of homeomorphisms $h_x$ of $\In$ with $[-x,x]$ to define a map 
$\iota_x$ which serves as a homotopy.  It is then elementary that the composite of $i$ with the homotopy inverse
of $q$ serves as a homotopy inverse to $\iota$.
\end{proof}

The main idea in the definition of $\iota$, and in particular the 
construction of $\tilde{\varphi}$, is ``focusing'' on the interval $\In = [-1,1]$ by modifying configurations in $\R$ 
with points outside the interval to have points at the endpoints of the interval instead.  Corresponding points in the configuration are replace by
standard points along the $x$-axis.  We will use similar ideas
 in the development of multiplicative structures on $AM_n$ and $\AM_n$.

\section{Multiplicative structures}
\label{LittleIntervals}

In this section we first define an action of the  little intervals operad $\C_1$ on the 
framed knot space and $AM_n^{fr}$, which
are compatible via the evaluation map.  
Then we construct a multiplication
 which is homotopy commutative on $\AM_n^{fr}$.  We show that these actions are all compatible up to homotopy.

\begin{definition}\label{C1definition}
The little intervals operad $\C_1$ has as its $n$th entry $\C_1(n)$ the space of $n$ disjoint subintervals of $\In$, 
which we topologize as a subspace of $\In^{2n}$ through the endpoints of the intervals.  

Given an subinterval $L \subseteq \In$ we by abuse let $L$ also denote the orientation-preserving affine-linear transformation which sends $\In$ to $L$. 

For use with operad actions on knot spaces, we let $\widehat{L}$ denote the map $\In^3 \to \In^3$ which applies $L$ to the first coordinate, 
and then shrinks the second and third coordinates by the same scaling factor (but doesn't translate them).
\end{definition}

\subsection{The little intervals action on the knot space}
\label{C1ActionOnTheKnotSpace}

If $\mathcal{L} = \bigcup L_i$ is a union of $k$ disjoint 
little intervals, its action on a $k$-tuple of
embeddings $f_i  : \R \to \R^3$ yields the embedding which at time $t$ has value
\begin{align*}
\mathcal{L}\cdot (f_1, \cdots, f_k)(t) = 
\left\{
\begin{array}{ll}
 \widehat{L_i} \circ f_i \circ {L_i}^{-1}(t)  & t \in {\rm{some}} \;\; L_i\\
(t, 0, 0) & \mbox{otherwise}. 
\end{array}
\right. 
\end{align*}
That is, the embeddings are ``shrunk and placed in succession'' according to $\mathcal{L}$.  
 The action on 
 $\Emb^{fr}(\R, \R^3)$ is similar, with the framings unchanged by the shrinking.   
 We will view this action as a case of ``insertion into the trivial embedding, with standard framing.''

\subsection{The spatial little intervals action on aligned maps} \label{SpatialAction}
Defining a $\C_1$ action on $AM_n$ is more involved, guided by wanting the evaluation map $ev_n$ to be compatible with the action. 
We want to take a
configuration in the interval and evaluate points in the various $L_i$ on different elements of $AM_n$. 
But typically fewer than $n$ points will be in each $L_i$,  so we adjust accordingly.  We first define restriction of an aligned map to some interval $L$.

\begin{definition}\label{restriction}
For $L \subset \In$, define $L^{-1}$ on some $\vec{t}$ in the interval by applying to each $t_i$ the piecewise-linear map which is 
$L^{-1}$ on the image of $L$, sends points to the left of $L$ to $-1$, and sends points to the right of $L$ to $1$.

Define the restriction of  $\phi \in AM_n$ to $L$, denoted $\phi |_L$, by applying $\phi$ to $L^{-1}(\vec{t}\;)$ and then applying 
projection maps to forget all of the
points in the resulting configurations whose indices $j$ do not correspond to a $t_j$ in the interior of $L$.
\end{definition}

 This is not continuous, as the different
projection maps depending on the number of $t_i$ in $L$ produce points in different configuration spaces, 
but it is an essential auxiliary construction.

\begin{definition}\label{IntervalsActionAM}
Define the action of a union of little intervals $\mathcal{L} = \bigcup L_i$ on a collection of maps $\phi_i \in AM_n$ as the map in $AM_n$ which takes $\vec{t} \in \Delta^n$ to the union of all of the $\widehat{L_i} \circ \phi_i |_{L_i}$ applied to $\vec{t}$ along with a point in the configuration at $(t_j, 0, 0)$ for each $t_j$ which is not in the interior of any $L_i$.  
If for such a $t_j$ we have, say, $t_j = t_{j+1}$, then we set the ``direction of collision from $(t_j,0,0)$ to $(t_{j+1},0,0)$'', denoted $v_{i,i+1}$ when we have used coordinates, to be the positive $x$-axis direction.  

The $\C_1$-action on $AM_n^{fr}$ is defined in the same manner, where  the framings are unchanged because each interval shrinks an aligned map equally
in all directions.  
\end{definition}

This action varies continuously as $t_j$ approaches an endpoint $e$ of some $L_i$ because the limit from either side of this endpoint is 
the configuration with the corresponding point $x_j$ at $(e, 0, 0)$.
Checking continuity elsewhere is immediate, as is checking the usual conditions required for an action of the operad $\C_1$.
The definition was arranged so that this action is compatible with the $\C_1$ action on the knot space via the evaluation map.

\begin{proposition}
\label{EvaluationIsC1Map}
$ev_n  : \Emb^{fr}(\R, \R^3)  \to  AM^{fr}_n$ is a map of $\C_1$-spaces.
\end{proposition}

For this paper we only need that this is a map of H-spaces.  
We will also need that the action is compatible with the structure maps in the Goodwillie-Weiss tower, which will be a main focus of 
Section~\ref{PropertiesOfProjMaps}. 
\subsection{A spatial-infinitesimal single little interval action}

In his work on operad actions on knot spaces \cite{Budney}, 
the first author extensively uses the fact that embeddings of $\R \times D^2$ into $\R \times D^2$ can 
be composed.  Our main idea in establishing homotopy commutativity of the multiplication(s) on  $AM_n^{fr}$ is to set  up  composition.  
In order to do so, we produce  products where the spatial points
are along the evaluation map of the unknot, and use infinitesimal composition for the essential part of the multiplication.
Given an interval $L$ and $\phi \in \AM^{fr}_n$, 
we produce an aligned map with infinitesimal configuration at the point $(L(0),0,0)$ together with some points along $\In \x (0,0)$.  For continuity, 
we need points near the boundary and, say, inside of $L$ to be pulled towards $L(0)$.

Let $L^\circ = L([-\frac{1}{2}, \frac{1}{2}])$, which is the ``core" of $L$.  
Let $e_L: \In \to \In$ be a monotone continuous map which sends 
$L^\circ$ to the point $L(0)$ and which is the identity outside of $L$.   By abuse use the same notation $e_L$ for the induced map on collections
of points in the interval.
Let $\iota(\vec{t}, L)$ be the configuration in $\In^3$ that has a point at $(L_i(0),0,0)$ for each $i=1,...,k$ and a point $(e(t_j),0,0)$ for each $t_j$ not contained in any $L$.  
%
%

As in previous constructions, we  define the map $L \cdot \phi$ piecewise on $\Delta^n$ according to the partition defined by which
indices of $t_i$
 occur in $\vec{t} \cap L^\circ$.  Let $i$ and $j$ be the indices of the leftmost and rightmost $t_k$ in $L^\circ$, which we will consider as functions
of $L^\circ$ and $\vec{t}$.  
For an element $\phi \in \AM^{fr}_n$, let $\phi|_L$ be defined similarly as in  Definition~\ref{restriction} for the $AM_n$ setting,
though this now produces {equivalence classes} of configurations module translation and scaling, with framings. 

\begin{definition}\label{FirstInfinitesimalIntervalAction}
With notation as above, for $\phi \in \AM_n$ define $L \cdot \phi \in AM_n$ by
\begin{equation*}
(L \cdot \phi) (\vec{t}) = 
(e_L( t_1,..., t_{i-1}, L(0), t_{j+1}, ..., t_n) \x (0,0)) \circ_{i(L^\circ, \vec{t})} \phi|_{L^\circ}(\vec{t}).
\end{equation*}
\end{definition}
For a fixed $L$, it is clear that the output varies continuously with $\phi$ and with all $\vec{t}$ which have the same indices occurring in $\vec{t} \cap L^\circ$.  
To check that the various pieces fit together to a continuous function, suppose that some $t_i$ is equal to $L(-\frac{1}{2})$, the left endpoint of $L^\circ$.  The formulae on the two pieces of $\Delta^n$ that meet at $t_i = L(-\frac{1}{2})$ are 
\[
((e_L(t_1, t_{i-1}, L(0), t_{j+1}, ...,t_n) \x(0,0)) \circ_i \phi((L^\circ)^{-1}(t_i,t_{i+1},...,t_j))
\]
and using that $e_L(t_i)=e_L(L(-\frac{1}{2}))=e_L(L(0)) = L(0)$,
\[
(e_L(t_1,..., t_{i-1}, L(0), L(0), t_{j+1}, ..., t_n) \x (0,0)) \circ_{i+1} \phi((L^\circ)^{-1}(t_{i+1},...,t_j)).
\]
The key point is that, by definition of the doubling maps in $\tC^{fr}_\bullet \la \In^3, \d \ra$, $\phi$ sends a point in $\d \In$ to a point 
``that looks infinitely far away from the images of the interior points in $\In$.''  Here it is essential that we are starting with elements of
the operadic mapping space model rather than the spatial model.

To elaborate, using our standard coordinates on these compactifications,
for any $k \in \{i+1,...,j\}$, $v_{ik}=(1,0,0)$ in either of the two configurations above, and every other $v_{k\ell}$ ($1\leq k\leq \ell \leq n$) 
is the same in these two configurations as well.  Furthermore, the projections to $(\In^3)^n$ of the two expressions agree, 
so they are the same configuration in $C^{fr}_n \la \In^3, \d \ra$.

Checking continuity between any other two pieces of $\Delta^n$ is similar, as is checking continuity if we vary both $\vec{t}$ and $L$.

If we compose $L \cdot \phi$ with  projection to $\widetilde{C}_n \la \R^3 \ra$, the resulting map 
satisfies the conditions of being in $\AM_n$.  We need the spatial information in the next section.


\subsection{A commutative multiplication on infinitesimal aligned maps}
\label{InfinitesimalAction}
Now we define a multiplication on $\AM^{fr}_n$ determined by a choice of two intervals $L_1,  L_2 \in \C_1(1)$, one entry in
the operad of ``overlapping intervals.''  Here when two intervals overlap, they are also ordered.  Informally, this order  says which interval is ``on top.''  
Since $\C_1(1)$ is connected,  this product will be homotopy-commutative.  


To see homotopy-commutativity of connect-sum holds for knots coherently 
as part of an operad action,
it is technically necessary to ``thicken'' the knots, by studying embeddings of $\R \times D^2$ in $\R \times D^2$.  
In \cite{Budney}, the first author shows that such embedding spaces have an action of the overlapping intervals operad, and 
moreover uses this action to determine the homotopy type
of the space of classical knots. As we work with evaluation maps between compactified configuration spaces, our substitute for $\R \times D^2$ is the following.

\begin{definition}
  Let $C_n^{it} \la \In \x (0,0), \d \ra$, the space of infinitesimally thickened configurations in an interval,  be the subspace  
 of $C_n^{fr} \la \In^3, \d \ra$ whose image under the projection to $(\In^3)^n$ lies in $(\In \x (0,0))^n$.
 \end{definition}
 
So far $L\cdot \phi$ is a map $C_n \la \In, \d \ra \to C_n^{fr} \la \In^3, \d \ra$ whose image lies in this subspace $C_n^{it} \la \In \x (0,0), \d \ra$.
 We can view $C_n \la \In, \d \ra$ as a subspace of $C_n^{it} \la \In \x (0,0), \d \ra$ in an obvious way with identity framings at every point. 
We will now extend the domain of $L\cdot \phi$ from $C_n \la \In, \d \ra$ to all of $C_n^{it} \la \In \x (0,0), \d \ra$.

For any $c \in C_n^{it} \la \In \x (0,0), \d \ra$, let $\vec{t}(c)=(t_1<...<t_m)$ be the set of \emph{distinct} points in $p(c)$ (so $m\leq n$).  Then $c$ can be written  as 
\begin{equation}
\label{ExpressConfig}
c = (...(\vec{t}(c) \x (0,0))\circ_{m_1} c_1)\circ_{m_2} ...)\circ_{m_k} c_k
\end{equation}
for some $c_i \in \tC_{n_i} \la \In^3 \ra$.  This expression is unique once 
we require that $m_i \geq m_{i-1} + n_{i-1}$, so that the underlying points of insertion are distinct.

\begin{definition}
Define the extension of $L \cdot \phi$ to $C_n^{it} \la \In \x (0,0), \d \ra$  as
\begin{equation*}
(L \cdot \phi)(c) = (...(((L\cdot \phi)(\vec{t}(c)))\circ_{m_1} c_1)\circ_{m_2} ...)\circ_{m_k} c_k.
\end{equation*}
Here the $L \cdot \phi$  on the right hand side is as in Definition~\ref{FirstInfinitesimalIntervalAction}, using the 
identification of  $C_n\la \In, \d \ra$ with $\Delta^n$.
\end{definition}

The key point now is checking continuity when points enter or exit the infinitesimal configurations $c_i$.  
The argument is just as for continuity in Definition~\ref{FirstInfinitesimalIntervalAction}
but with  $(1,0,0)$ replaced by the tangent vector, which is the first vector in the framing, at $\phi(t_{m_i})$.

Now both the input and output of the map $L \cdot \phi$ can be regarded as elements in $C_n^{it} \la \In \x (0,0), \d \ra$.  Thus two such maps can be composed, and we denote the composition by $\circ$.

\begin{definition}
\label{SecondInfinitesimalIntervalAction}
Given two little intervals $L_1, L_2 \in \C_1(1)$,  define the product of elements $\phi, \psi \in \AM^{fr}_n$ as 
\begin{equation*}
\label{ActionOfTwoOverlappingIntervals}
\mu_{L_1, L_2} (\phi, \psi)) : \vec{t} \mapsto ((L_2 \cdot \psi) \circ (L_1 \cdot \phi)) (\vec{t}),
\end{equation*}
where we take the equivalence class in $\tC^{fr}_n \la \In^3 \ra$ of the right-hand side.

Let $L_-=[-1,0], L_+=[0,1]$.  Abbreviate $\mu_{L_-, L_+}$ as $\mu$, and abbreviate $\mu(\phi, \psi)$ as $\phi \cdot \psi$.
\end{definition}

\begin{theorem}
\label{HtpyCommThm}
 $\phi \cdot \psi$ is homotopic to $\psi \cdot \phi$.
\end{theorem}
\begin{proof}
The map $(L, \phi) \mapsto (L \cdot \phi)$ varies continuously with $L$.  Because
$\C_1(1)$ is connected,
 any two multiplications induced by choices of $(L_1, L_2)$ --
in particular $(L_-, L_+)$ and $(L_+, L_-)$ -- are homotopic. \end{proof}

\subsection{A compatible little intervals action on infinitesimal aligned maps}

We now compare the multiplications on the spatial model $AM_n^{fr}$ and the infinitesimal model $\AM_n^{fr}$.  We first present a simpler 
multiplication $\mu'$ on $\AM_n^{fr}$, which is homotopic to the one defined above.  The multiplication $\mu'$ avoids the use of the maps 
$e_L$ which pull points towards $L$, which are needed homotopy-commutativity.  The multiplication $\mu'$ also has the advantage of extending 
to an action of the little intervals operad.

\begin{definition}
\label{MuMultiplication}
Let $\mu'$ be the map 
$$
\mu' : \C_1(2) \x (\AM_n^{fr})^2 \to \AM_n^{fr}
$$
defined by 
$$
\mu'_{L_1,L_2}(\phi,\psi))(\vec{t}):=
\left( \left( (t_1,...,t_{i-1}, L_1(0), t_{j+1},...,t_{k-1}, L_2(0), t_{\ell+1},...,t_n)\x (0,0) \right) \circ_i \phi|_{L_1}(\vec{t}) \right) \circ_k \psi|_{L_2} (\vec{t})
$$
where $t_i, ...t_j$ are the points in $L_1$, and $t_k,...,t_\ell$ are the points in $L_2$.
Similarly, for any $k\geq 1$, define a map
\begin{equation}
\label{C1ActionOnInfConfigs}
\C_1(k) \x (\AM_n^{fr})^k \to \AM_n^{fr}
\end{equation}
which sends $((L_1,...,L_k),(\phi_1,...,\phi_k))$ to the result of inserting each $\phi_i|_{L_i}$ into the point $(L_i(0),0,0)$.
\end{definition}

It is straightforward to verify that the maps (\ref{C1ActionOnInfConfigs}) define a $\C_1$-action on $\AM_n^{fr}$, using the fact that $\widetilde{C}_n\la \In^3\ra$ 
records only directions of collision, as well as again using the fact that in the operadic models configuration points which corresponding to times $t_i$ which are
equal to $\pm 1$ are ``at infinity.''

\begin{proposition} 
\label{HtpicToMu}
Let $L_1, L_2$ be disjoint intervals.
The multiplication $\mu_{L_1, L_2}(\phi, \psi)$ is homotopic to the multiplication $\mu'_{L_1, L_2}(\phi, \psi)$.  
\end{proposition}

\begin{proof}
Since configurations of points along the $x$-axis are all equivalent in $\widetilde{C}_n\la \In^3 \ra$, the resulting aligned maps differ only in 
the subsets of $\In$ on which they are constant.  The aligned map resulting from $\mu'$ sends all $t_j$ between $L_1, L_2$ to the same 
configuration point, while the same is true for $\mu$ and $L_1^\circ, L_2^\circ$.
These are related by a homotopy which ``reparametrizes the domains'' of $\phi, \psi$ from $L_1^\circ, L_2^\circ$ to $L_1, L_2$.
\end{proof}



Recall from Definition~\ref{Iota} our transformation $\iota$ from the spatial to the operadic mapping space models, which is an equivalence. 

\begin{proposition}
\label{CompatibleMultiplications}
The $\C_1$-actions on $AM_n^{fr}$ and $\AM_n^{fr}$ are compatible.  That is, the diagram
\[
\xymatrix{
\C_1(k) \x (AM^{fr}_n)^k \ar[r] \ar[d]_-{\iota^k} & AM^{fr}_n \ar[d]^-{\iota} \\
\C_1(k) \x (\AM^{fr}_{n})^k \ar[r] & \AM^{fr}_{n} 
}
\]
commutes up to homotopy, where the top horizontal map is the action given in Definition~\ref{IntervalsActionAM} and the bottom horizontal map is the action (generalizing the multiplication $\mu'$ given in Definition \ref{MuMultiplication}).
\end{proposition}
\begin{proof}
Recall that the $\C_1$ action on $AM^{fr}_n$ is defined by taking the union of the images of $\widehat{L_i} \circ \phi_i |_{L_i}$, along with points along the $x$-axis.  
We will perform a homotopy with three steps, parametrized by $s \in [0,3]$,
from the composite through the upper-right corner to the composite through the lower-left corner.  

The first step, as $s$ varies in $[0,1]$ is a straightforward homotopy from $(L_1, ..., L_k) \cdot (\phi_1,...,\phi_k)$ to $(L_1^\circ, ..., L_k^\circ) \cdot (\phi_1,...,\phi_k)$, where as before $L^\circ := L([-\frac{1}{2}, \frac{1}{2}])$.  The configuration produced by an aligned map in the image at $s=1$ is the union of configurations resulting from all the $\phi_i|_{L_i^\circ}$, together with points along the $x$-axis.  Roughly, we are ensuring that each aligned map is standard on some interval between each pair of intervals $L_i, L_{i+1}$, so that we can push apart and shrink configurations continuously.

Next define $L_i^\x := L_i \setminus L_i^\circ$.  In the second step, as $s$ varies in $[1,2]$, we scale the image of each $\phi_i|_{L_i^\circ}$ to an infinitesimal configuration at $(L_i(0),0,0)$.  In the notation of Section \ref{C1ActionOnTheKnotSpace}, we follow each aligned map by $(\widehat{L_i^\circ})^{-1} \circ\widehat{J_s}\circ \widehat{L_i^\circ}$, where we define $J_s:=[s-2,2-s]$ for $s\in [1,2]$.  
This can be done at the level of representatives in $C_n^{fr}\la \In^3, \d \ra$ coming from the top horizontal map.
The configurations become infinitesimal only at $s=2$, so continuity requires us to simultaneously pull the images of $L_i^\x$ toward $(L_i(0),0,0)$, from occupying $L_i^\x \x \{(0,0)\}$ at $s=1$ to occupying all of $L_i \x \{(0,0)\}$ at $s=2$.  

At $s=2$, the elements in $\AM_n^{fr}$ produce infinitesimal configurations $\phi_i|_{L_i^\circ}$, together with points on the $x$-axis between them (where the distances between them are \emph{not} recorded by the operadic compactification $\tC^{fr}_n\la \In^3\ra$).
This aligned map looks almost like the composite through the lower-left corner.  
The only difference is that this aligned map is constant on each component complementary to the $L_i^\circ$, rather than each component complementary to the $L_i$.
The last stage of the homotopy thus simply requires ``reparametrizing the domain'' of each $\phi_i$ from $L_i^\circ$ to all of $L_i$, as in the proof of Proposition \ref{HtpicToMu}.
\end{proof}

Because $\iota$ is a homotopy equivalence, we have the following.

\begin{corollary}
\label{CommutativeMultOnAM}
The $\C_1$ action on $AM_n^{fr}$ induces a homotopy-commutative multiplication.
\end{corollary}

We have not  used that the ambient dimension is three, so these results hold for knots in higher dimensional Euclidean spaces as well. 
Similar results were proven in \cite{SinhaOperads} for knots modulo immersions and would work
similarly for framed knots, 
but only for the {limit} of the $AM_n^{fr}$.  Recently Turchin \cite{VictorDelooping} has established a version of this theorem, 
along with  group structure as in Theorem~\ref{Pi0IsGroup} below,
for the stages in the tower in the cosimplicial model.   Dwyer and Hess have similar results for the limit \cite{DwyerHess}.
We were not able to show that the structure 
studied in Turchin's paper is compatible with the evaluation map,  which necessitated the present approach.

\begin{remark}
At this point we can explain the connection between Schubert's elementary geometric result that connect 
sum of knots is commutative \cite{Schubert} and Steenrod's deep, formal work on commutativity of cup 
product and operations in cohomology \cite{Steenrod}.  This connection was implied to us by the work of 
McClure and Smith \cite{MSMultivariable, MSCubes}, as applied in this setting by Sinha \cite{SinhaOperads},
whose product structures on totalizations of cosimplicial spaces are related to Steenrod's formulae for
higher cup products \cite{MSMultivariable}.  

In different notation from Definition~\ref{IntervalsActionAM}, the product of two aligned maps $f$ and $g$ is 
$$ (t_1, \cdots, t_n) \mapsto \bigcup_{t_i \leq 0 \leq t_{i+1}} \hat{f}(t_1, \cdots, t_i) * \hat{g}(t_{i+1}, \cdots t_n).$$
Here the union refers to a decomposition of the domain, the $n$-simplex; 
$\hat{f}$ and $\hat{g}$ are obtained from $f$ and $g$ by appending times and rescaling (as we did 
regularly this section); and $*$ indicates a ``stacking product'' of configurations in $\R^d$.

This is formally similar to the standard formula for cup product
$$ \phi \cup \psi (\sigma : [v_1, \cdots, v_n] \to X) = \sum \phi(\sigma|_{[v_1, \cdots v_i]}) 
\cdot \psi (\sigma|_{[v_{i+1}, \cdots, v_n]}).$$
Here $\phi$ and $\psi$ are cochains, $\sigma$ is a chain, and brackets $[\ ,\ ]$ around some variables refers
to the simplex given as convex linear combinations of those variables.  This sum is zero but for one term,
but as McClure and Smith show, it is the correct sum to write down for purposes of generalization.

Recall that Theorem~\ref{HtpyCommThm}  gives a homotopy from $\mu(f,g)$ to the multiplication defined by 
(\ref{ActionOfTwoOverlappingIntervals}) with $(L_1,L_2)=([-1,0],[0,1])$.  
By choosing a path in $\C_1(1)$ from $([-1,0],[0,1])$ to $([0,1],[-1,0])$ 
(and applying a homotopy from Proposition \ref{HtpicToMu}) we ultimately get a homotopy from $\mu(f,g)$ to $\mu(g,f)$. 
We choose the following path, which in the overlapping intervals setting the first interval always lies above the second.
Start with $\left([-1,0], [0, 1] \right)$; grow the second interval to obtain $\left( [-1, 0], [-1, 1] \right)$; then translate $[-1,0]$ to $[0,1]$; 
finally shrink $[-1,1]$ to $[-1,0]$ to obtain the pair $\left( [0,1], [-1,0] \right)$.

If we apply the formulae (\ref{ActionOfTwoOverlappingIntervals}) for the products of $f$ and $g$ governed by this path of $1$-disks,
we see a formal analogue for Steenrod's formula for cup-one, namely
$$ \phi \cup_1 \psi (\sigma) = \sum_{i<j} \phi(\sigma|_{[v_1, \cdots, v_i, v_j, \cdots, v_n]} \cdot 
															\psi (\sigma|_{[v_i, \cdots, v_j]}).$$
The main difference is that the product rather than using an underlying commutative ring 
uses operad insertion maps.
McClure and Smith show this to be an appropriate extension of Steenrod's formula to the operadic setting.
\end{remark}

\section{Maps and layers in the tower, and abelian group structure }
\label{PropertiesOfProjMaps}

 
Our goal is to show that each stage of the Goodwillie--Weiss tower for knots has an
abelian group structure compatible with connected sum.  Though we are primarily interested in mapping space models, 
we use the cosimplicial models and language around them as a starting point and key organizational tool.  Cosimplicial structures are also
essential for the study of spectral sequences below.  For our applications, we develop a variety of models for the maps in the totalization
tower of a cosimplicial space.

\subsection{Maps in the tower through projection and restriction}

In the cosimplicial realm the totalization tower is a basic object of study, dual in a sense to the skeletal filtration of a simplicial complex.
When using the functors $\mathcal{G}_n$ to pull back a subcubical diagram from a cosimplicial space, the maps from  $\ttot^n$ to $\ttot^{n-1}$
are induced by inclusions $\P_\nu[n-1] \to \P_\nu[n]$.  

In the mapping space models, this inclusion of indexing categories explicitly gives rise to the following.

\begin{definition}\label{TowerMap}
The restriction-projection $p_n : AM_n^{fr} \to AM_{n-1}^{fr}$ sends a map $\phi \in AM_n^{fr}$ to the composite 
\[
\xymatrix{
\Delta^{n-1} \ar[r]^-d & \Delta^n \ar[r]^-\phi & C_n^{fr} \la \In^3, \d \ra \ar[r]^-s & C^{fr}_{n-1} \la \In, \d \ra
}\]
where $d$ and $s$ are the images of a (face, degeneracy) pair whose composite is the identity.  
As any two such choices of a pair $(d, s)$ yield homotopic projections $AM^{fr}_n \to AM^{fr}_{n-1}$,
we take $d$ to be the map $d_n: (t_1,...,t_{n-1}) \mapsto (t_1,...,t_{n-1},1)$ and $s$ to be the map $s_n$ that forgets the last configuration point and framing.  

The restriction-projection map $\widetilde{p_n} : \AM^{fr}_n \to \AM^{fr}_{n-1}$ and non-framed versions are defined analogously.
\end{definition}

Then $p_n$ is our first model for the standard map  $\ttot^n C_\bullet^{fr} \la \In^d, \d \ra $ to $\ttot^{n-1} C_\bullet^{fr} \la \In^d, \d \ra $, and its main use is the following.

\begin{proposition}
\label{ProjectionIsC1Map}
The restriction-projection $p_n : AM^{fr}_n \to AM^{fr}_{n-1}$ is a map of $\C_1$-algebras.
\end{proposition}

\begin{figure}
\begin{picture}(350, 450)
\put(60,0){\includegraphics[scale=0.4]{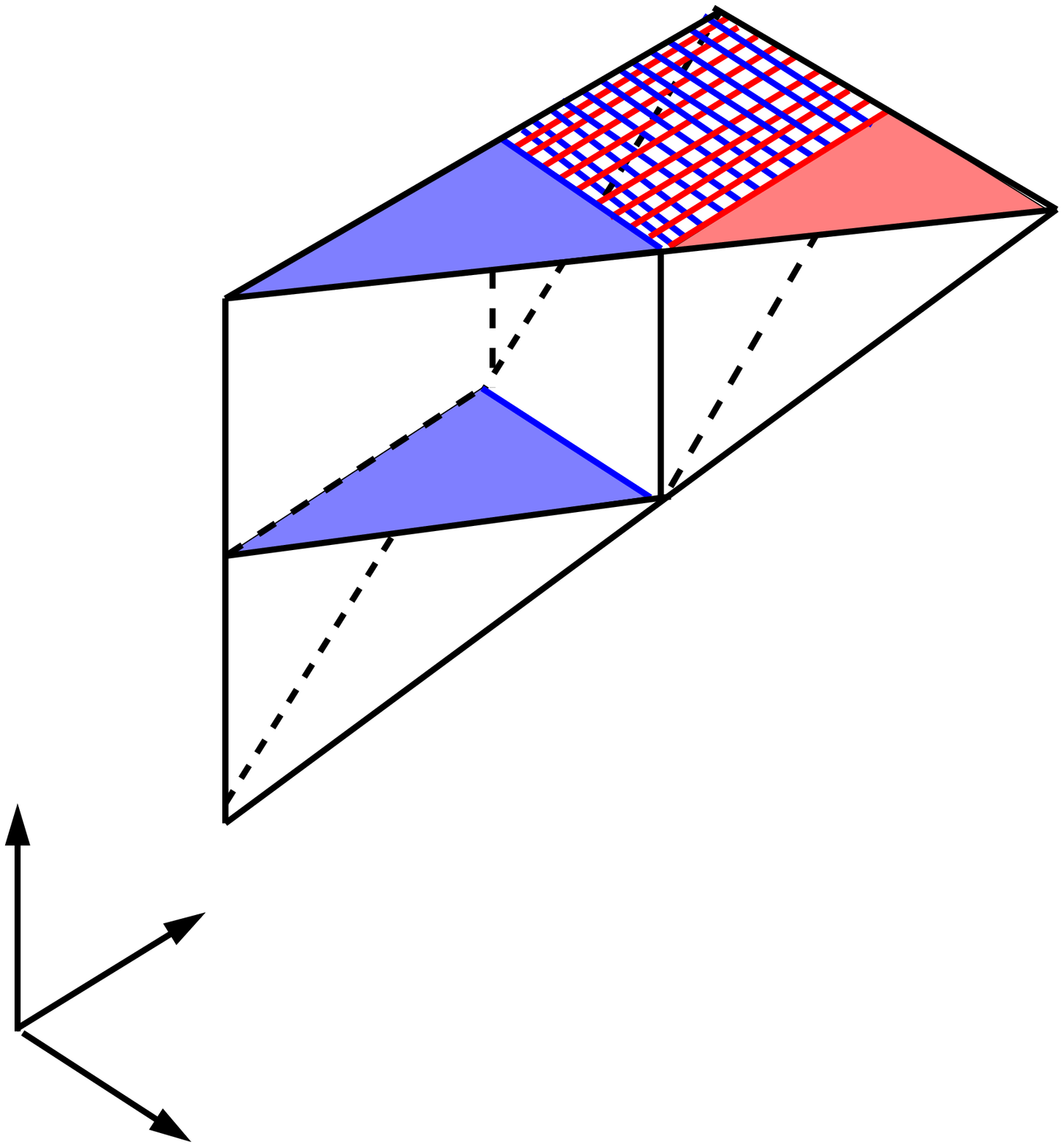}}
\put(100,0){$t_1$}
\put(102,46){$t_2$}
\put(60,70){$t_3$}
\end{picture}

\caption{A picture which partially illustrates how the projection $AM_n \to AM_{n-1}$ 
preserves the $\C_1$-action for $n=3, k=2$, with intervals $(L_-, L_+)=([-1,0],[0,1])$ acting on $\phi, \psi \in AM_n$.  
The product $(L_-,L_+)\cdot (\phi, \psi)$ is a map of the whole simplex.  
The restriction to each light-blue triangular face is (after rescaling in the domain and range, 
and after forgetting the last configuration point) the projection of $\phi$ to $AM_2$; 
the blue edge of this triangle is in turn the restriction to the face $t_2=1$.  
Similarly the restriction to the light-red triangular face is the projection of $\psi$ to $AM_2$; 
the red edge of this triangle is in turn the restriction to the face $t_1=0$.  
The map on the top face of the whole simplex is the projection of $(L_-, L_+)\cdot (\phi, \psi)$, 
which is indeed obtained as a product of the projections of $\phi$ and $\psi$ to $AM_2$.}

\end{figure}

\begin{proof}
We need to check that the following diagram commutes:
\[
\xymatrix{
\C_1(k) \x (AM^{fr}_n)^k \ar[r] \ar[d] & AM^{fr}_n \ar[d] \\
\C_1(k) \x (AM^{fr}_{n-1})^k \ar[r] & AM^{fr}_{n-1} 
}
\]
The composite through the upper-right corner is the map 
\[
(t_1,...,t_{n-1}) \,\, \longmapsto \,\, s_n \left( \bigcup_{i=1}^k \widehat{L_i} \circ \phi_i \circ L_i^{-1}(t_1,...,t_{n-1},1) \cup \bigcup \{ (t_j,0,0) \} \right)
\]
while the composite through the lower-left corner is the map 
\[
(t_1,...,t_{n-1}) \,\, \longmapsto \,\, \bigcup_{i=1}^k \widehat{L_i} \circ s_n \circ (\phi_i|_{d_n \Delta^{n-1}}) \circ L_i^{-1}(t_1,...,t_{n-1}) \cup \bigcup \{ (t_j,0,0) \}.
\]
These maps agree.
In either expression, each $\phi_i$ is applied to the same configuration, in particular with the same number of $t_i =1$.
As for the resulting configuration in $\R^3$, in the first expression one forgets an extra framed point at $(1,0,0)$.
In the second expression, one forgets an extra framed point at $(L_i(1),0,0)$ for each $i=1,...,k$, yielding the same result.
\end{proof}
A similar argument shows that the projection $AM_n \to AM_{n-1}$ is a map of $\C_1$-algebras.

\subsection{Maps in the tower through (only) projection}

For purposes of connecting our multiplications with those coming from cosimplicial structure,  we need  another model for the maps and 
layers in the totalization tower, one which is closely related to $p_n$  in that it uses projection.
While the map $p_n$ is defined through including $\P_\nu[n-1]$ in $\P_\nu[n]$, we now proceed by ``fibering $\P_\nu[n]$ over $\P_\nu[n-1]$.''
Consider the functor $i_n$ from $\P_\nu[n]$ to $\P_\nu[n-1]$ which modifies a subset $S$ by identifying $(n-1)$ and $n$ -- 
that is, changing  occurrences of $n$ in $S$ to $n-1$.

\begin{lemma}\label{AnotherTowerModel}
The homotopy limit of $X^\bullet \circ \mathcal{G}_{n-1} \circ i_n$ is homotopy equivalent to $\widetilde{Tot}^{n-1} X^\bullet$.\\

The $n$th degeneracy $s^n : X^n \to X^{n-1}$ extends to a map  
$$\widehat{s^n} : X^\bullet \circ \mathcal{G}_n \to X^\bullet \circ \mathcal{G}_{n-1} \circ i_n.$$  On homotopy limits, the map
induced by $\widehat{s^n}$ agrees with the standard map from 
$\widetilde{Tot}^n X^\bullet \to \widetilde{Tot}^{n-1} X^\bullet$.
\end{lemma}

\begin{proof}
For the first statement, it suffices to show that the functor $i_n$ is left cofinal \cite[XI.9.2]{BK},
so that the category $P_\nu [n] \times_{P_\nu [n-1]} (P_\nu [n -1] \downarrow S)$  has contractible 
nerve for every $S \in  P_\nu [n-1]$. But this category has a final object, namely the
object corresponding to $S \cup {n} \in P_\nu [n]$. 

It follows from the cosimplicial identities that $s^n$ induces a natural transformation of functors 
$\widehat{s^n} :  X^\bullet \circ \mathcal{G}_n \to X^\bullet \circ \mathcal{G}_{n-1} \circ i_n$, 
where the map on each object is given by either its last degeneracy or the identity map. 

Again, the standard map $\widetilde{Tot}^n X^\bullet \to \widetilde{Tot}^{n-1} X^\bullet$ is induced by the 
inclusion functor $j_n : \P_\nu [n-1] \to \P_\nu [n]$.  Note that the composite $i_n \circ j_n$ is the identity.   
Consider the following diagram to complete the proof.
\[
\xymatrix{
\holim (X^\bullet \circ \mathcal{G}_n) \ar[d]_-{-\circ j_n} \ar[dr]^-{\widehat{s^n}} & \\
\holim  (X^\bullet \circ \mathcal{G}_{n-1}) &  \holim (X \circ \mathcal{G}_{n-1} \circ i_n) \ar[l]^-{-\circ j_n}_-\simeq
}
\]
The left-hand vertical map is the map $\widetilde{Tot}^n X^\bullet \to \widetilde{Tot}^{n-1} X^\bullet$,
since $X^\bullet \circ \mathcal{G}_n \circ j_n = X^\bullet \circ \mathcal{G}_{n-1}$.  
Precomposing $X^\bullet \circ \mathcal{G}_{n-1} \circ i_n$ with $j_n$  induces the horizontal map because $i_n \circ j_n = id$.  
This horizontal map is an equivalence because on these homotopy limits, ``precomposing with $j_n$" is the 
right-inverse to the equivalence given by ``precomposing with $i_n$."
\end{proof}

By abuse, we let $\widehat{s^n}$ also denote the map it induces on homotopy limits.  For purposes
of distinction we let ${AM^{fr}_{n-1}}(\Delta^n)$ denote $C^{fr}_\bullet \la \In^3, \d \ra \circ \mathcal{G}_n \circ i_n$, which by the above is a model for $AM^{fr}_{n-1}$.  (The notation indicates that we model $AM_{n-1}^{fr}$ by maps of $\Delta^n$.)
The fact that  $\widehat{s^n} : AM^{fr}_n \to {AM^{fr}_{n-1}}(\Delta^n)$ and $p_n$ both model the 
structure maps in the totalization tower implies that they are homotopic.  
We choose $\widehat{s^n}$ for more detailed study of the fibers of these maps, starting with the observation
that $\widehat{s^n}$ is a fibration in both the $\AM^{fr}_n$ and $AM_n^{fr}$ settings, 
since the projection or identity maps which define it entry-wise are fibrations.
(See \cite[Lemma 3.5]{BCSS} for an explicit proof in this case of a standard result about enriched model structures on diagram categories in general.)

\begin{definition}
Let $L_n$ be the fiber of $\widehat{s^n} : AM^{fr}_n \to {AM^{fr}_{n-1}}(\Delta^n)$, based at the evaluation map of the unknot.  
That is, $L_n$ is the space of aligned maps where when one forgets the last point in each configuration in the image one obtains a standard configuration along the $x$-axis parametrized by the points in the
domain simplex.  
Let $\La_n$ be the fiber of  $\widehat{s^n} : \AM^{fr}_n \to {\AM^{fr}_{n-1}}(\Delta^n)$, 
which then sits over the constant map at the infinitesimal configuration where all $x_{ij}$ for $i < j$ are equal to $(1,0,0)$.

We will write $L_n(\Delta^n)$ or $\La_n(\Delta^n)$ if we want to emphasize the model we are using for this fiber.
\end{definition}

We used both the $AM_n$ and $\AM_n$ models in the previous section
since  since the former supports an evaluation map and a $\C_1$ structure and the latter has a commutative
multiplication.  We use both $L_n$ and $\La_n$ in similar fashion here, which means we also need a comparison.

\begin{proposition}\label{LayerLemma1}
The map $\iota : AM_n \to \AM_n$ restricts to a map from $L_n$ to $\La_n$ which is an equivalence and preserves multiplication up to homotopy.
\end{proposition}

\begin{proposition}\label{InclusionIsC1}
The inclusion of $L_n$ into $AM_n$ is a map of $\C_1$-spaces. 
\end{proposition} 

The proofs of all of these are straightforward from the definitions, checking that previous definitions arguments such as 
Proposition~\ref{CompatibleMultiplications}
are compatible with the
condition of being a standard configuration but for the last coordinate.  

We can use the equivalence ${AM_{n-1}^{fr}}(\Delta^n) \to  AM^{fr}_{n-1}$ defined by restriction to the last face of $\Delta^n$
to define a $\C_1$-structure on ${AM_{n-1}^{fr}}(\Delta^n)$.  Combining 
Proposition~\ref{ProjectionIsC1Map} and Proposition~\ref{InclusionIsC1} and composing with the homotopy inverse 
$AM^{fr}_{n-1} \to {AM_{n-1}^{fr}}(\Delta^n)$, which is then also $\C_1$,  gives the following.

\begin{corollary}\label{C1Fibration}
$L_n \to AM_n^{fr} \to {AM_{n-1}^{fr}}(\Delta^n)$ is a fibration sequence of $\C_1$ spaces.
\end{corollary}

\subsection{The layers of a totalization tower via cubical diagrams} \label{CosimplicialCubical}

Following Goodwillie, we use cubical diagrams to  develop loop structures on layers in the totalization tower of a cosimplicial space.  
Some of this material is treated  in the forthcoming volume \cite{MunsonVolicCubes}.  We have already seen $\P_\nu[n]$
and $\P_\nu[n-1]$ related by inclusion and ``fibering''.  We next relate them through a  Mayer-Vietoris decomposition.

Just as the standard map ${\ttot}^n X^\bullet \to {\ttot}^{n-1} X^\bullet$ can be defined by an inclusion  
$\Delta_{n-1} \subset \Delta_n$, it can also be defined by the canonical inclusion 
$\P_\nu[n-1] \subset \P_\nu[n]$ through the equivalence given by the functor $\mathcal{G}_n$.
We will analyze the fibers of these maps -- that is, the layers in the totalization tower -- in two different ways this section.  
First we use a sort of Mayer--Vietoris decomposition of  the category $\P_\nu[n]$.

\begin{definition}
\begin{itemize}
\item Let $\P_{\neq n}  \subset \P_\nu[n]$ be the full subcategory given by all nonempty subsets of $[n]$ except the singleton $\{n\}$. \\
\item Let $\P_{n \in}$ be the (cubical) poset of all subsets of $[n]$ containing $n$.  \\
\item Let $\P_{n+}$ be the (subcubical) poset of subsets of $[n]$ containing $n$ and at least one other element.
\end{itemize}
\end{definition}

The inclusion $\iota: \P_\nu[n-1] \incl \P_{\neq n} $ is left cofinal, so the map induced on homotopy limits is an equivalence.
We can thus replace  $\holim_{ \P_\nu[n]} X^\bullet \circ \G_n \to \holim_{ \P_\nu[n-1]} X^\bullet \circ \G_n$ by an
alternate model for the maps in the tower, namely
\begin{align*}
\holim_{\P_\nu[n]} X^\bullet \circ \G_n \to \holim_{\P_{\neq n}} X^\bullet \circ \G_n. 
\end{align*}

The poset $\P_\nu[n]$ can be written as the union of $\P_{\neq n}$ and $\P_{n\in}$ along  $\P_{n+}$, yielding the following square:
\begin{align}
\label{SquareOfPosets}
\xymatrix{
  \P_\nu[n] & \P_{n\in} \ar@{_(->}[l] \\
  \P_{\neq n} \ar@{^(->}[u] & \P_{n+} \ar@{_(->}[l] \ar@{^(->}[u]
  }
\end{align}
Applying $\holim_{(-)} X^\bullet \circ \G_n$ to the diagram above, we get a pullback square of fibrations \cite[Proposition 0.2]{GwCalc2}. 
Thus, to study the fiber(s) of the map from $\widetilde{{\rm Tot}}^n$ to $\widetilde{{\rm Tot}}^{n-1}$, which up to homotopy is the left-hand column of the induced map of homotopy limits of this square, it suffices to study the right-hand column.  We say fiber(s) because in our application we study unbased and sometimes
disconnected spaces.  

Since $\P_{n+}$ is just the cube $\P_{n\in}$ with its initial object removed, the map on homotopy limits induced by the
right vertical arrow is 
just the map from the initial object, at $\{ n \}$, to the homotopy limit of the rest of the diagram, which is subcubical.  
We conclude that the fiber(s) are  the total fiber(s) of the cube $\P_{n\in}$, over different possible basepoints in the unbased cased.  
 In the based setting the original map is $k$-connected if and only if the $n$-cube $X^\bullet \circ \mathcal{G}_n|_{\P_{n\in}}$ 
is $k$-cartesian. 

\subsection{Loop structure on layers via retracts of cubes}

\begin{lemma}\label{RetractOfCubes}
Let $f: C_\bullet \to D_\bullet$ be a map of cubical diagrams, and let $r: D_\bullet \to C_\bullet$ be a retraction object-wise.
Then the total fiber of the cube $\left(C_\bullet \to D_\bullet\right)$ is loops on the total fiber of $\left(D_\bullet \to C_\bullet\right)$.
\end{lemma}

\begin{proof}
Consider the square of cubes, which is itself a cubical diagram:
\[ 
\xymatrix{
C_\bullet \ar[r]^-{f} \ar[d]_-{id} & D_\bullet \ar[d]^-{r}\\
C_\bullet \ar[r]^-{id} & C_\bullet
}
\]
We find the total fiber of this cubes in two ways, first internally to the $C_\bullet$ and $D_\bullet$ subcubes followed by taking horizontal then vertical
fibers.  This yields the total fiber of $\left(C_\bullet \to D_\bullet\right)$.  If we first take internal fibers, then vertical, then horizontal we see 
loops on the total fiber of $\left(D_\bullet \to C_\bullet\right)$.
\end{proof}

Cosimplicial identities imply that the codegeneracy maps of a cosimplicial space can be used to define retractions of cubes.  
For example, at the first two levels of a cosimplicial space the codegeneracy
map is a retract for either coface map.  Since the first layer in the Tot tower is the total fiber of $\P_{1 \in}$, which is just a coface map $X^0 \to X^1$,
this lemma shows that is loops on the fiber of the codegeneracy $X^1 \to X^0$.  More generally we have the following.

\begin{definition}
\begin{itemize}
\item For an inclusion of ordered sets $i : S \hookrightarrow S'$, we define the {dual surjection} $i^! : S' \to S$
to be the order-preserving retraction which sends each element of $S$ to the maximal value of $S'$ possible among such retractions.

\item Let $\P_{n\in}^!$ be the category whose objects are subsets of $[n]$ containing $n$ and where morphisms are all the dual surjections.

\item  For brevity, let ${\mathcal{G}}_n^{!} : \P_{n\in}^! \to  \bf{\Delta}_{n}$ denote the functor ${\mathcal{G}}_{\P_{n\in}^!}$ (defined in Definition~\ref{GFunctor}).

\end{itemize}
\end{definition}

\begin{proposition} \label{CubicalFibers}
For a cosimplicial space $X^\bullet$, the homotopy limit of $X^\bullet \circ \mathcal{G}_n|_{\P_{n\in}}$ and thus the fiber of
${\ttot}^n X^\bullet \to {\ttot}^{n-1} X^\bullet$ is homotopy equivalent to $\Omega^n \holim X^\bullet \circ \mathcal{G}_n^!$.
\end{proposition}

The proof of this proposition in \cite[Proposition 9.4.10]{MunsonVolicCubes} is essentially correct, though it makes reference to a diagram which 
is generally not possible to construct for an arbitrary cosimplicial space and in particular for those we use.  We only use a sub-diagram of that
used \cite{MunsonVolicCubes} which can always be constructed.

\begin{proof}
We interpolate between $\mathcal{G}_{\P_{n\in}}$ and $\mathcal{G}_{\P_{n\in}^!}$, which share the underlying objects, namely subsets 
of $[n]$ which contain $n$.  But $\P_{n\in}$ has basic morphisms (of which all others are composites) of $S \mapsto S \cup i$, while 
$\P_{n\in}^!$ has basic morphisms with some $i, i + 1 \mapsto i+1$ and identity otherwise.  
For $j=0,1,...,n$, define $\P_{n\in}(j)$ as having this same set of objects but with generating morphisms $S \subset S \cup i$ for $i \leq j$ and 
$i, i + 1 \mapsto i+1$ and identity otherwise for $i > j$.

We view $X^\bullet \circ {\mathcal{G}}_{\P_{n \in}(j+1)}$ as a map of cubes 
$C_\bullet \to D_\bullet$ where $C_\bullet$ is the restriction to subsets which do not contain $j$,
 $D_\bullet$ is the restriction to those which do, and the map between cubes is defined by all of the $S \mapsto S \cup j$ maps.  
 Then $X^\bullet \circ {\mathcal{G}}_{\P_{n \in}(j)}$ is a retract 
 $D_\bullet \to C_\bullet$, with morphisms defined by sending $j$ to the next element in the ordering.
We deduce from Lemma~\ref{RetractOfCubes} that the total fiber of $X^\bullet \circ \mathcal{G}_{\P_{n \in}(j+1)}$ is 
loops on the total fiber of $X^\bullet \circ \mathcal{G}_{\P_{n \in}(j)}$.  Thus 
${\rm fib} X^\bullet \circ \mathcal{G}_{\P_{n \in}(n)} \simeq \Omega^n {\rm fib}  X^\bullet \circ \mathcal{G}_{\P_{n \in}(0)}$.
Since $ \mathcal{G}_n|_{\P_{n\in}}$ is $\mathcal{G}_{\P_{n\in}(n)}$ while $\mathcal{G}_{\P_{n\in}^!}$ is ${\mathcal{G}}_{\P_{n\in}(0)}$, we obtain the result.
\end{proof}

\begin{remark}
It is a tautology that the $n$-th layer in the totalization tower of a fibrant cosimplicial space is an $n$-fold loop space.  These cubical models for both the entries and the layers of the tower give a workable approach to this $n$-fold loop structure in the non-fibrant setting.
\end{remark}

\subsection{Surjectivity on components of maps in the tower}

In order to inductively establish a group structure on components of stages in the tower, we need a surjectivity result.

\begin{theorem}
\label{SurjOnPi0}  
The  restriction-projection map $\tilde{p_n}: \AM_n \to \AM_{n-1}$ induces a surjection on $\pi_0$. 
\end{theorem}

We prove this unframed version first and then use it to prove the desired framed version, as
the end of the proof we give here breaks down in the framed setting.

\begin{proof}[Proof of Theorem~\ref{SurjOnPi0}]

We extend techniques from the previous subsection. Recall from Section~\ref{TheModels} 
the cosimplicial space $\tC'_\bullet \la \In^3\ra$, which we abbreviate here as $C$.
This functor defines both our cosimplicial model and, through pullback by the functors $\mathcal{G}_n$, our operadic mapping space model.

We use the square of posets (\ref{SquareOfPosets}) above.  
We are led to the map on homotopy limits induced by  $\P_{n+} \incl \P_{n\in},$
which if surjective on components implies the same for the map $\AM_n \to \AM_{n-1}$,
the left-vertical map in the square (\ref{SquareOfPosets}).

The initial object in $\P_{n\in}$ is $C (\{n\}) = \widetilde{C}'_0 \la \R^3 \ra$, a point.  
Thus the map on homotopy limits induced by $\P_{n+} \incl \P_{n\in}$  is surjective on components 
if and only if the homotopy limit over $\P_{n+}$ is connected.  
As before we reindex this diagram, using the isomorphism of $\P_{n+}$ with 
$\P_\nu[n-1]$ which sends a set containing $n$ to the set obtained by removing $n$.

We  prove the connectedness of this homotopy limit by induction on $n$.  
The  case $n=1$ is immediate, as  $\widetilde{C}'_{1}\la \R^3 \ra = S^2$ is connected.
For the induction step, we exhibit the homotopy limit $\P_\nu[n-1]$ as a fibration over a connected space with a connected fiber.  
Consider the reindexed pushout square (\ref{SquareOfPosets}),  with $n$ replaced by $n-1$ everywhere.
The inclusion $\P_\nu[n-2] \incl \P_{\neq n-1}$ is left cofinal, so the left-hand column of the reindexed (\ref{SquareOfPosets}) gives a 
fibration whose base, the homotopy limit of $D$ which we can take over
$\P_\nu[n-2]$, is connected by induction.  

The square of holim's induced by (\ref{SquareOfPosets}) is a pullback,
so it suffices to establish connectedness of the fiber of the map induced by the right-hand column,
taken over the component to which the connected space $\holim_{\P_{\neq n-1}} D$ maps.  
Here we choose basepoints for $D$ by choosing the basepoint $(1,0,0)$ in $S^2 = \widetilde{C}'_1\la \In^3  \ra$.  

This induced square of holim's is of based spaces, we can describe the
fiber of this map as the total fiber of the based cube $D(\P_{n-1})$.  Apply Proposition~\ref{CubicalFibers} 
to deduce that this total fiber is $\Omega^{n-1}\mathrm{tfib} D^!$.

To show that $\Omega^{n-1}\mathrm{tfib} D^!$ is connected or, equivalently, that $D^!$ is $(n-1)$-Cartesian, we 
use a Blakers-Massey Theorem.
This is an $(n-1)$-cube of spaces $\tC'_i \la \In \ra$, $i \leq n-1$, of configurations in $\In^3$ up to scaling and translation, with a tangent vector at each point.  The maps forget points and corresponding tangent vectors.  We  replace this by a homotopy equivalent cube of spaces $C'_i \la \In^3 \ra$ or even 
$C'_i \left( \In^3 \right)$ of  configurations in $\In^3$ with a tangent vector at each point.  
Every map in this cube is a fibration, so we can take the fiber in one direction. 
The resulting $(n-2)$-cube which we call $\phi D^!$ has entries $\In^3 - f([i])$, where the deleted points are
images of a fixed embedding $f: [n] \hookrightarrow \In^3$.  The maps in the cube are inclusions of 
 open submanifolds and are thus cofibrations.   Moreover, $\phi D^!$
 is a push-out cube, so it is strongly co-Cartesian.  
 
 Each  inclusion
 $\In^3 -  f([i+1]) \hookrightarrow \In^3 - f([i])$ is a $2$-connected map.
 The Blakers-Massey theorem (Theorem~2.3 of \cite{GwCalc2}; see also \cite[Theorem 6.2.1]{MunsonVolicCubes})
 applies to give that the total fiber of $\phi D^{!}$  is $n$-connected.  Thus its $n$th loop space is connected, which yields the result.
\end{proof}

We will say more about this total fiber -- and thus the layers in the tower -- both below in this section and in Section~\ref{SpectralSequence} when we make spectral sequence calculations.

In the proof just given, the analogous $(n-1)$-cube with frames instead of tangent vectors is not $(n-1)$-Cartesian, which is why we prove 
the framed version separately now.

\begin{theorem}
\label{SurjOnPi0Framed}
In the framed setting, the map $\widetilde{p_n}: \AM^{fr}_n \to \AM^{fr}_{n-1}$ induces a surjection on $\pi_0$. 
\end{theorem}
\begin{proof}
 As $\AM^{fr}_n$ is a subspace of $\Map(\Delta^n, \tC_n\la \In^3 \ra \x O(3)^n)$, we consider the diagram
\[
\xymatrix{
\Delta^n \ar[r] \ar[d] & \tC_n\la \In^3 \ra \x O(3)^n \ar[d]  \ar[r]  & \tC_n \la \In^3 \ra \x (S^2)^n \ar[d] \\
\Delta^{n-1} \ar[r] & \tC_{n-1} \la \In^3 \ra \x O(3)^{n-1} \ar[r]  & \tC_{n-1} \la \In^3 \ra \x (S^2)^{n-1}.
}
\]
Suppose we have $\Phi \in AM^{fr}_{n-1}$.  
Let $\phi$ be the image of $\Phi$ under the map $AM^{fr}_{n-1} \to AM_{n-1}$, which essentially composes a map to $\tC_n\la \In^3 \ra \x O(3)^n$ with the projection onto $\tC_n\la \In^3 \ra \x (S^2)^n$, using the first vector in each frame.  
By Theorem \ref{SurjOnPi0}, there is a $\psi \in AM_n$ whose image in $AM_{n-1}$ is in the same component as $\phi$.  We  
 lift $\psi$ to a map $\Delta^n \to \tC_n\la \In^3 \ra \x O(3)^{n-1}$, using $\Phi$ to define the map to the $O(3)^{n-1}$ factor.  
 
It remains to lift this map to one additional factor of $O(3)$.  
The map $\Delta^n \to \tC_n\la \In^3 \ra \x O(3)^n$ is prescribed on two faces of $\Delta^n$, namely the faces
on which the additional factor must agree with another factor.  
Away from these faces, there 
are no constraints on the map to the additional  factor of $O(3)$.  Thus topologically the problem is  to 
extend this map from $D^{n-1} \subset \d D^n$ to $D^n$, which is immediate.
\end{proof}

\subsection{Group structure}
\label{GroupStructure}

\begin{lemma}
\label{LoopspaceStructuresAgree}
The multiplication on $\La_n$ obtained by restricting the multiplication 
$\mu'$ of Definition~\ref{MuMultiplication}
is homotopic to the one coming from the description of $\La_n$ as the $n$-fold loopspace of the total fiber of an $n$-cube of configuration spaces in Proposition~\ref{CubicalFibers}.
\end{lemma}

This compatibility is key in proving the following.  Recall the multiplication $\mu$ from Definition~\ref{SecondInfinitesimalIntervalAction}. 

\begin{theorem}
\label{Pi0IsGroup}
$\pi_0 (\AM_n^{fr})$ is an abelian group with the multiplication $\mu$ (or $\mu'$) on $\AM_n^{fr}$.
\end{theorem}

\begin{proof}

By Theorem~\ref{HtpyCommThm}, $\pi_0(\AM^{fr}_n)$ is an abelian monoid.  We prove that this monoid is  a group by induction on $n$.  

The  case  $n=0$ is clear, since $\AM_0^{fr} = *$.  Suppose $n\geq 1$, and $\AM_{n-1}^{fr}$ is a group.  
From Lemma~\ref{LoopspaceStructuresAgree}, the fiber $\widetilde{L}_n$ is a loop space  when 
$n\geq 1$, so that its components form a group.
Moreover, $\pi_0(\widetilde{L}_n) \to \pi_0(\AM_n)$ is a map of monoids.
Thus applying $\pi_0$ to the fiber sequence $\widetilde{L}_n \to \AM^{fr}_n \to \AM_{n-1}^{fr}$  gives an exact sequence of monoids.
By Theorem \ref{SurjOnPi0Framed}, $\pi_0(\AM^{fr}_n \to \AM_{n-1}^{fr})$ is surjective.  
The  lemma now follows from the elementary fact that if $G \to H \to K \to 0$ is an exact sequence of monoids 
with $G, K$ groups, then $H$ is a group.

\end{proof}

\begin{corollary}
The homotopy fibers of $AM^{fr}_n$ over varying components of $AM^{fr}_{n-1}$ are homotopy equivalent.
\qed
\end{corollary}

\begin{proof}[Proof of Lemma~\ref{LoopspaceStructuresAgree}]

Explicitly, ${\La_n}(\Delta^n)$ is the subspace of $\AM_n$  of maps $\Delta^n \to \tC'_n \la \In^3 \ra$ whose projection by 
forgetting the last point in the configuration yields the chosen constant map.  
This implies that such maps are themselves constant on the $t_n = t_{n-1}$ and $t_n = t_{n+1} (= 1)$ faces of $\Delta^n$.
We show that $\mu'$ and the multiplication from the $n$-fold loopspace structure both agree with another multiplication, which we now define.

Consider the map $\Delta^{n-1} \x \In \to \Delta^n$ given by $(t_1,...,t_{n-1}) \x t \mapsto (t_1,...,t_{n-1}, (1-t) t_{n-1} + t)$.  By pre-composing by this map, we can view an element of ${\La_n}(\Delta^n)$ as a loop in $\Map(\Delta^{n-1}, \widetilde{C}'_n \la \In^3 \ra)$, based at the constant standard map.
As we use configuration spaces modulo translations and scalings, these loops begin and end at the exact same map.  Loop concatenation 
then defines a  multiplication $\mu_\Omega$ on ${\La_n}(\Delta^n)$, which of course has homotopy inverses.

We first prove that  $\mu_\Omega$ is homotopic to the multiplication $\mu'$ of Definition~\ref{MuMultiplication} by ``homotoping away the mixed terms of  $\mu'$.''  Consider the function 
$r : \Delta^n \times \In \to \Delta^n$ which linearly interpolates between a configuration of times 
$\vec{t}=(t_1,...,t_n)$ and the ``constant configuration" at $t_n$,
but only until all the $t_i$ are contained in $[-1, 0]$ or $[0,1]$.  Let $r_s$ denote its restriction to $\Delta^n \times s$.
This function  $r$ is not continuous at points with $t_n=0$.
But because a product $\mu(\phi, \psi)$ of elements $\phi, \psi \in {\La_n}(\Delta^n)$ is constant along $t_n=0$, the  map
given by 
\[
\vec{t} \mapsto
\left\{
\begin{array}{ll}
\phi(r_s(\vec{t})) & \mbox{if } t_n \leq 0 \\
\psi(r_s(\vec{t})) & \mbox{if } t_n \geq 0 
\end{array}
\right.
\]
 is ultimately continuous.
Moreover, for any $s$ the resulting map is again an element of $\AM_n$ and in fact ${\La_n}(\Delta^n)$.
More specifically, in the image of the product $\mu'(\phi, \psi)$ in 
$\widetilde{C}_n\la \In^d \ra \subset \prod_{i < j} (S^{d-1})$,  the
vectors labelled by pairs $i,j$ with $t_i \leq 0$ and $t_j \geq 0$ are all of the form $(1, 0, 0)$, and this will be preserved throughout
the homotopy.

Thus the ``mixed stacking terms'' in the product are not essential and up to homotopy the product 
can be ``decoupled'' to a map which at each fixed time uses only one of $\phi$ and $\psi$.

The multiplication when $s=1$ agrees with $\mu_\Omega$, as the map $\Delta^{n-1} \x I \to \Delta^n$ that we used to define $\mu_\Omega$ sends $\Delta^{n-1} \x [-1,0]$ to $\{ t_n \leq 0\}$ and $\Delta^{n-1} \x [0,1]$ to $\{ t_n \geq 0\}$.  Thus $\mu'$ restricted to ${\La_n}(\Delta^n)$ 
is homotopic to $\mu_\Omega$, which completes the first half of our proof.

Next, Proposition \ref{CubicalFibers} in the case of $\AM_n$ implies that ${\La_n}(\Delta^n)$ 
is homotopy equivalent to $\Omega^n \tfib (\tC'_\bullet \la \In^3\ra \circ \G_n^!)$,
where $\tC'_\bullet \la \In^3\ra \circ \G_n^!$ is a cube of configuration spaces with structure maps which forget points.  
This is established in the proof of Proposition~\ref{CubicalFibers} through a series of equivalences 
$$
\tfib(\tC'_\bullet\la \In^3\ra \circ \mathcal{G}_n)  \simeq 
\Omega \tfib(\tC'_\bullet\la \In^3\ra \circ \mathcal{G}_{\P_{n\in}(n-1)}) \simeq \cdots 
\simeq \Omega^i \tfib(\tC'_\bullet\la \In^3\ra \circ \mathcal{G}_{\P_{n\in}(n-i)}).
$$

The multiplication coming from the single loopspace structure on $\Omega \tfib(\tC'_\bullet\la \In^3\ra \circ \mathcal{G}_{\P_{n\in}(n-1)}$
agrees up to homotopy with the others coming from  $i$-fold loopspace structure, including $i=n$. 
We will complete the proof by showing that this single loopspace multiplication coincides with $\mu_\Omega$.  

By definition $\tfib(\tC'_\bullet\la \In^3\ra \circ \mathcal{G}_{\P_{n\in}(n-1)}$ is the total fiber of an $(n-1)$-cube of fibers, 
where each fibration forgets the last point.   
An element of loops on this total fiber is a loop 
of maps $\alpha:I^{n-1} \to \fib\left(\tC'_n\la \In^3\ra \to  \tC'_{n-1}\la \In^3\ra\right)$ where
\begin{itemize}
\item each map $\alpha$ is constant on the face $\{t_i=1\}$ for every $i$ and
\item  on each of the remaining $(n-1)$ faces, the image of $\alpha$ is in the image of the appropriate doubling map.
\end{itemize}
On the other hand, an element of ${\La_n}(\Delta^n)$ is a loop
of maps $\alpha: \Delta^{n-1} \to \tC'_n\la \In^3\ra$ where
\begin{itemize}
\item the image of each $\alpha$ in $\tC'_{n-1}\la \In^3\ra$ is the standard constant map,
\item $\alpha$ is constant on the face $\{t_{n-1}=1\}$, and
\item on each of the remaining $(n-1)$ faces, the image of $\alpha$ is in the image of the appropriate doubling map.
\end{itemize}

Consider a homeomorphism from $I^{n-1}$ to $\Delta^{n-1}$ which identifies 
$\bigcup_{i=1}^n \{t_i=1\} \subset I^{n-1}$ with $\{t_{n-1}=1\} \subset \Delta^{n-1}$.
Via this homeomorphism, we get a homeomorphism between loops on this total fiber and 
${\La_n}(\Delta^n)$ which is compatible with the multiplications on each.
\end{proof}

\subsection{Summary}

Putting results together the results, the sequences $L_n \to AM_n^{fr} \to AM^{fr}_{n-1}$  in the Goodwillie-Weiss
tower approximating classical framed knots satisfy the following.  
\begin{enumerate}
\item They are fibration sequences of $\C_1$-spaces (Definition~\ref{IntervalsActionAM} and Corollary~\ref{C1Fibration}).  
\item $AM_n^{fr}$ receives a multiplication-preserving (in fact, $\C_1$-action-preserving) evaluation map from the knot space, with connect sum as its multiplication (Proposition~\ref{EvaluationIsC1Map}).
\item At the level of components, all multiplications are commutative (Corollary \ref{CommutativeMultOnAM}) and have inverses (Theorem \ref{Pi0IsGroup} and Proposition \ref{CompatibleMultiplications}).  Moreover,  $AM_n^{fr} \to AM^{fr}_{n-1}$
is surjective on components (Theorem~\ref{SurjOnPi0Framed}).
\end{enumerate}

\section{The homotopy tower is a finite-type invariant}
\label{FiniteType}

In this section we show that $\pi_0 (ev_n) \colon \pi_0 \Emb^{fr}(\R, \R^3) \to \pi_0 AM_n^{fr}$, which we've shown to be an abelian group valued invariant 
compatible with connect-sum of knots,  is a finite-type invariant of type $(n-1)$. 
The main tool is a theorem of Habiro \cite{Habiro2000}, which states that two classical knots  share finite-type invariants of degree $\leq n-1$, if and only if they differ by a series of {$C_n$-moves.} 

We describe these moves  in a way that will facilitate our proof. Let $E_2$ be a copy of $D^2\x I = D^2 \x [-1,1]$, with two properly embedded subarcs which clasp in the center as in Figure~\ref{fig:clasps}. 
Iteratively form $E_n$ from $E_{n-1}$ by replacing a regular neighborhood of the 
top left arc of $E_n$ by a copy of $E_2$.  Thus $E_n$ is $D^2 \x I$ with $n$ properly embedded arcs.
(So technically each $E_n$ is a pair of spaces.)

\begin{figure}[h]
\begin{center}
\begin{minipage}{4.5cm}
\begin{tikzpicture}[scale=2]
\coordinate(a) at (0,0);
\coordinate(b) at (2,0);
\coordinate(c) at (0,-1);
\coordinate(d) at (2,-1);
\coordinate(f) at (1.3,-.5);
\coordinate(g) at (.7,-.5);
\draw[name path=line 1](a) to[out=0, in =90] (f);
\path[name path=line 2](c) to[out=0,in=270] (f);
\path[name path=line 3](b) to[out=180, in =90] (g);
\draw[name path=line 4](d) to[out=180,in=270] (g);
\path [name intersections={of = line 1 and line 3}];
 \coordinate (S)  at (intersection-1);
\path [name intersections={of = line 2 and line 4}];
 \coordinate (T)  at (intersection-1);
 \filldraw[color=white] (S) circle (.1);
 \filldraw[color=white] (T) circle (.1);
\draw(b) to[out=180, in =90] (g);
\draw(c) to[out=0,in=270] (f);
\end{tikzpicture}
\end{minipage}
\hspace{1em}
\begin{minipage}{5.5cm}
\begin{tikzpicture}[scale=1]
\coordinate(a) at (-.5,0);
\coordinate(b) at (-.5,1);
\coordinate(c) at (-.5,2);
\coordinate(d) at (-.5,2.25);
\coordinate(e) at (-.5,2.75);
\coordinate(f) at (-.5,3);
\coordinate(g) at (4.5,0);
\coordinate(h) at (4.5,3);
\coordinate(C1) at (1,1.5);
\coordinate(C2) at (2,2.2);
\coordinate(C3) at (2.2,1.2);
\coordinate(C4) at (2,1.2);
\coordinate(D1) at (1.7,0);
\coordinate(D2) at (1,1);
\coordinate(E1) at (1.7,3);
\coordinate(E2) at (1.7,2.75);
\coordinate(E3) at (1.0,2.25);
\coordinate(E4) at (.8,2);
\coordinate(E5) at (2.4,2.25);
\coordinate(E6) at (2.65,2.25);
\coordinate(E7) at (1.8,1.25);
\coordinate(E8) at (1.8,1);
\coordinate(A1) at (2.2,4);
\coordinate(A2) at (2.2,-1);
\coordinate(A3) at (3.7,1.5);
\path[name path=line 1](h) to[out=180,in=90] (C1) to[out=270,in=180] (g);
\path[name path=line 2](a) to[out=0,in=180] (D1) to[out=0,in=0] (C2) to[out=180,in=0]  (D2) to (b);
\path[name path=line 3] (f) to (E1) to[out=0,in=90]  (E6) to[out=270,in=270] (C4) to[out=90,in=270] (E5) to[out=90,in=0] (E2) to (e);
\path[name path=line 4] (d) to (E3) to[out=0,in=180] (E7) to[out=0,in=90] (C3) to[out=270,in=0] (E8) to[out=180,in=0] (E4) to (c);
\path[name path=line 5] (A1) to[out=270,in=90] (A3) to[out=270,in=90]  (A2);
\path [name intersections={of = line 1 and line 2}];
 \coordinate (U1)  at (intersection-1);
 \coordinate (U2) at (intersection-2);
\path [name intersections={of = line 1 and line 3}];
 \coordinate (U3)  at (intersection-1);
 \coordinate (U4) at (intersection-2);
 \path [name intersections={of = line 1 and line 4}];
 \coordinate (U5)  at (intersection-1);
 \coordinate (U6) at (intersection-2);
\path [name intersections={of = line 2 and line 3}];
 \coordinate (U7)  at (intersection-1);
 \coordinate (U8) at (intersection-2);
\path [name intersections={of = line 2 and line 4}];
 \coordinate (U9)  at (intersection-1);
 \coordinate (U10) at (intersection-2); 
 \path [name intersections={of = line 3 and line 4}];
 \coordinate (U11)  at (intersection-1);
 \coordinate (U12) at (intersection-2);
 \draw(C1) to[out=270,in=180] (g);
 \filldraw[color=white] (U1) circle (.1);
 \filldraw[color=white] (U2) circle (.1);
 \draw(C2) to[out=180,in=0]  (D2) to (b);
  \filldraw[color=white] (U9) circle (.07);
   \filldraw[color=white] (U10) circle (.07);
\draw (d) to (E3) to[out=0,in=180] (E7) to[out=0,in=90] (C3);
    \filldraw[color=white] (U5) circle (.1);
    \filldraw[color=white] (U12) circle (.05); 
\draw (C4) to[out=90,in=270] (E5) to[out=90,in=0] (E2) to (e);
    \filldraw[color=white] (U4) circle (.1);
    \filldraw[color=white] (U8) circle (.1);   
\draw(f) to (E1) to[out=0,in=90]  (E6) to[out=270,in=270] (C4);
    \filldraw[color=white] (U3) circle (.1);
    \filldraw[color=white] (U7) circle (.1);   
    \filldraw[color=white] (U11) circle (.02);      
 \draw(C3) to[out=270,in=0] (E8) to[out=180,in=0] (E4) to (c);
  \filldraw[color=white] (U6) circle (.1);
\draw (a) to[out=0,in=180] (D1) to[out=0,in=0] (C2);
 \draw(h) to[out=180,in=90] (C1) ; 
\end{tikzpicture}
\end{minipage}
\end{center}
\caption{The basic clasp $E_2$ and the iteratively constructed $E_4$. The ambient cylinders are not drawn.}\label{fig:clasps}
\end{figure}
 Now a basic $C_n$-move on a knot $K$ is given by finding an embedding $e$ of $E_n$ into $\R^3$ which meets the knot $K$ as the given collection of arcs, and sliding another subarc of $K$ across the central disk 
$e(D^2 \x \{0\})$ 
of the embedded $E_n$ as in Figure~\ref{fig:nmove}.  In this Figure, the almost vertical strand in each left (respectively right) picture is a subarc of the knot which is isotopic to the front (respectively back) half of the boundary of this central disk.
\begin{figure}[h]
$$
\begin{minipage}{2cm}
\begin{tikzpicture}[scale=1]
\coordinate(a) at (0,0);
\coordinate(b) at (2,0);
\coordinate(c) at (0,-1);
\coordinate(d) at (2,-1);
\coordinate(f) at (1.3,-.5);
\coordinate(g) at (.7,-.5);
\coordinate(h) at (1.2,.5);
\coordinate(i) at (1.2,-1.5);
\coordinate(j) at (1.8,-.5);
\draw[name path=line 1](a) to[out=0, in =90] (f);
\path[name path=line 2](c) to[out=0,in=270] (f);
\path[name path=line 3](b) to[out=180, in =90] (g);
\draw[name path=line 4](d) to[out=180,in=270] (g);
\path [name intersections={of = line 1 and line 3}];
 \coordinate (S)  at (intersection-1);
\path [name intersections={of = line 2 and line 4}];
 \coordinate (T)  at (intersection-1);
 \filldraw[color=white] (S) circle (.1);
 \filldraw[color=white] (T) circle (.1);
\draw(b) to[out=180, in =90] (g);
\draw(c) to[out=0,in=270] (f);
\path[name path=line 5] (h) to[out=-90,in=90] (j) to[out=-90,in=90] (i);
\path [name intersections={of = line 5 and line 3}];
 \coordinate (U1)  at (intersection-1);
\path [name intersections={of = line 5 and line 4}];
 \coordinate (U2)  at (intersection-1);
 \filldraw[color=white] (U1) circle (.1);
 \filldraw[color=white] (U2) circle (.1);
\draw(h) to[out=-90,in=90] (j) to[out=-90,in=90] (i);
\end{tikzpicture}
\end{minipage}
\,\,
\mapsto
\,\,
\begin{minipage}{2cm}
\begin{tikzpicture}[scale=1]
\coordinate(a) at (0,0);
\coordinate(b) at (2,0);
\coordinate(c) at (0,-1);
\coordinate(d) at (2,-1);
\coordinate(f) at (1.3,-.5);
\coordinate(g) at (.7,-.5);
\coordinate(h) at (1.2,.5);
\coordinate(i) at (1.2,-1.5);
\coordinate(j) at (1.8,-.5);
\path[name path=line 1](a) to[out=0, in =90] (f);
\path[name path=line 2](c) to[out=0,in=270] (f);
\path[name path=line 3](b) to[out=180, in =90] (g);
\path[name path=line 4](d) to[out=180,in=270] (g);
\draw[name path=line 5] (h) to[out=-90,in=90] (j) to[out=-90,in=90] (i);
\path [name intersections={of = line 5 and line 3}];
 \coordinate (U1)  at (intersection-1);
\path [name intersections={of = line 5 and line 4}];
 \coordinate (U2)  at (intersection-1);
 \filldraw[color=white] (U1) circle (.1);
 \filldraw[color=white] (U2) circle (.1);
 \draw(d) to[out=180,in=270] (g);
\draw(a) to[out=0, in =90] (f);
\path [name intersections={of = line 1 and line 3}];
 \coordinate (S)  at (intersection-1);
\path [name intersections={of = line 2 and line 4}];
 \coordinate (T)  at (intersection-1);
 \filldraw[color=white] (S) circle (.1);
 \filldraw[color=white] (T) circle (.1);
\draw(b) to[out=180, in =90] (g);
 \draw(c) to[out=0,in=270] (f);
\end{tikzpicture}
\end{minipage}
$$
$$
\begin{minipage}{5cm}
\begin{tikzpicture}[scale=1]
\coordinate(a) at (-.5,0);
\coordinate(b) at (-.5,1);
\coordinate(c) at (-.5,2);
\coordinate(d) at (-.5,2.25);
\coordinate(e) at (-.5,2.75);
\coordinate(f) at (-.5,3);
\coordinate(g) at (4.5,0);
\coordinate(h) at (4.5,3);
\coordinate(C1) at (1,1.5);
\coordinate(C2) at (2,2.2);
\coordinate(C3) at (2.2,1.2);
\coordinate(C4) at (2,1.2);
\coordinate(D1) at (1.7,0);
\coordinate(D2) at (1,1);
\coordinate(E1) at (1.7,3);
\coordinate(E2) at (1.7,2.75);
\coordinate(E3) at (1.0,2.25);
\coordinate(E4) at (.8,2);
\coordinate(E5) at (2.4,2.25);
\coordinate(E6) at (2.65,2.25);
\coordinate(E7) at (1.8,1.25);
\coordinate(E8) at (1.8,1);
\coordinate(A1) at (2.2,4);
\coordinate(A2) at (2.2,-1);
\coordinate(A3) at (3.7,1.5);
\path[name path=line 1](h) to[out=180,in=90] (C1) to[out=270,in=180] (g);
\path[name path=line 2](a) to[out=0,in=180] (D1) to[out=0,in=0] (C2) to[out=180,in=0]  (D2) to (b);
\path[name path=line 3] (f) to (E1) to[out=0,in=90]  (E6) to[out=270,in=270] (C4) to[out=90,in=270] (E5) to[out=90,in=0] (E2) to (e);
\path[name path=line 4] (d) to (E3) to[out=0,in=180] (E7) to[out=0,in=90] (C3) to[out=270,in=0] (E8) to[out=180,in=0] (E4) to (c);
\path[name path=line 5] (A1) to[out=270,in=90] (A3) to[out=270,in=90]  (A2);
\path [name intersections={of = line 1 and line 2}];
 \coordinate (U1)  at (intersection-1);
 \coordinate (U2) at (intersection-2);
\path [name intersections={of = line 1 and line 3}];
 \coordinate (U3)  at (intersection-1);
 \coordinate (U4) at (intersection-2);
 \path [name intersections={of = line 1 and line 4}];
 \coordinate (U5)  at (intersection-1);
 \coordinate (U6) at (intersection-2);
\path [name intersections={of = line 2 and line 3}];
 \coordinate (U7)  at (intersection-1);
 \coordinate (U8) at (intersection-2);
\path [name intersections={of = line 2 and line 4}];
 \coordinate (U9)  at (intersection-1);
 \coordinate (U10) at (intersection-2); 
 \path [name intersections={of = line 3 and line 4}];
 \coordinate (U11)  at (intersection-1);
 \coordinate (U12) at (intersection-2);
 \path [name intersections={of = line 1 and line 5}];
 \coordinate (V1)  at (intersection-1);
 \coordinate (V2) at (intersection-2);
 \draw(C1) to[out=270,in=180] (g);
 \filldraw[color=white] (U1) circle (.1);
 \filldraw[color=white] (U2) circle (.1);
\filldraw[color=white] (V2) circle (.1);
 \draw(C2) to[out=180,in=0]  (D2) to (b);
  \filldraw[color=white] (U9) circle (.07);
   \filldraw[color=white] (U10) circle (.07);
\draw (d) to (E3) to[out=0,in=180] (E7) to[out=0,in=90] (C3);
    \filldraw[color=white] (U5) circle (.1);
    \filldraw[color=white] (U12) circle (.05); 
\draw (C4) to[out=90,in=270] (E5) to[out=90,in=0] (E2) to (e);
    \filldraw[color=white] (U4) circle (.1);
    \filldraw[color=white] (U8) circle (.1);   
\draw(f) to (E1) to[out=0,in=90]  (E6) to[out=270,in=270] (C4);
    \filldraw[color=white] (U3) circle (.1);
    \filldraw[color=white] (U7) circle (.1);   
    \filldraw[color=white] (U11) circle (.02);      
 \draw(C3) to[out=270,in=0] (E8) to[out=180,in=0] (E4) to (c);
  \filldraw[color=white] (U6) circle (.1);
\draw (a) to[out=0,in=180] (D1) to[out=0,in=0] (C2);
 \draw(h) to[out=180,in=90] (C1) ; 
 \filldraw[color=white] (V1) circle (.1);
  \draw (A1) to[out=270,in=90] (A3) to[out=270,in=90]  (A2);
\end{tikzpicture}
\end{minipage}
\,\,
\mapsto
\,\,
\begin{minipage}{5cm}
\begin{tikzpicture}[scale=1]
\coordinate(a) at (-.5,0);
\coordinate(b) at (-.5,1);
\coordinate(c) at (-.5,2);
\coordinate(d) at (-.5,2.25);
\coordinate(e) at (-.5,2.75);
\coordinate(f) at (-.5,3);
\coordinate(g) at (4.5,0);
\coordinate(h) at (4.5,3);
\coordinate(C1) at (1,1.5);
\coordinate(C2) at (2,2.2);
\coordinate(C3) at (2.2,1.2);
\coordinate(C4) at (2,1.2);
\coordinate(D1) at (1.7,0);
\coordinate(D2) at (1,1);
\coordinate(E1) at (1.7,3);
\coordinate(E2) at (1.7,2.75);
\coordinate(E3) at (1.0,2.25);
\coordinate(E4) at (.8,2);
\coordinate(E5) at (2.4,2.25);
\coordinate(E6) at (2.65,2.25);
\coordinate(E7) at (1.8,1.25);
\coordinate(E8) at (1.8,1);
\coordinate(A1) at (2.2,4);
\coordinate(A2) at (2.2,-1);
\coordinate(A3) at (3.7,1.5);
\path[name path=line 1](h) to[out=180,in=90] (C1) to[out=270,in=180] (g);
\path[name path=line 2](a) to[out=0,in=180] (D1) to[out=0,in=0] (C2) to[out=180,in=0]  (D2) to (b);
\path[name path=line 3] (f) to (E1) to[out=0,in=90]  (E6) to[out=270,in=270] (C4) to[out=90,in=270] (E5) to[out=90,in=0] (E2) to (e);
\path[name path=line 4] (d) to (E3) to[out=0,in=180] (E7) to[out=0,in=90] (C3) to[out=270,in=0] (E8) to[out=180,in=0] (E4) to (c);
\path[name path=line 5] (A1) to[out=270,in=90] (A3) to[out=270,in=90]  (A2);
\path [name intersections={of = line 1 and line 2}];
 \coordinate (U1)  at (intersection-1);
 \coordinate (U2) at (intersection-2);
\path [name intersections={of = line 1 and line 3}];
 \coordinate (U3)  at (intersection-1);
 \coordinate (U4) at (intersection-2);
 \path [name intersections={of = line 1 and line 4}];
 \coordinate (U5)  at (intersection-1);
 \coordinate (U6) at (intersection-2);
\path [name intersections={of = line 2 and line 3}];
 \coordinate (U7)  at (intersection-1);
 \coordinate (U8) at (intersection-2);
\path [name intersections={of = line 2 and line 4}];
 \coordinate (U9)  at (intersection-1);
 \coordinate (U10) at (intersection-2); 
 \path [name intersections={of = line 3 and line 4}];
 \coordinate (U11)  at (intersection-1);
 \coordinate (U12) at (intersection-2);
 \path [name intersections={of = line 1 and line 5}];
 \coordinate (V1)  at (intersection-1);
 \coordinate (V2) at (intersection-2);
 \draw (A1) to[out=270,in=90] (A3) to[out=270,in=90]  (A2);
 \filldraw[color=white] (V1) circle (.1);
 \filldraw[color=white] (V2) circle (.1); 
 \draw(C1) to[out=270,in=180] (g);
 \filldraw[color=white] (U1) circle (.1);
 \filldraw[color=white] (U2) circle (.1);
 \draw(C2) to[out=180,in=0]  (D2) to (b);
  \filldraw[color=white] (U9) circle (.07);
   \filldraw[color=white] (U10) circle (.07);
\draw (d) to (E3) to[out=0,in=180] (E7) to[out=0,in=90] (C3);
    \filldraw[color=white] (U5) circle (.1);
    \filldraw[color=white] (U12) circle (.05); 
\draw (C4) to[out=90,in=270] (E5) to[out=90,in=0] (E2) to (e);
    \filldraw[color=white] (U4) circle (.1);
    \filldraw[color=white] (U8) circle (.1);   
\draw(f) to (E1) to[out=0,in=90]  (E6) to[out=270,in=270] (C4);
    \filldraw[color=white] (U3) circle (.1);
    \filldraw[color=white] (U7) circle (.1);   
    \filldraw[color=white] (U11) circle (.02);      
 \draw(C3) to[out=270,in=0] (E8) to[out=180,in=0] (E4) to (c);
  \filldraw[color=white] (U6) circle (.1);
\draw (a) to[out=0,in=180] (D1) to[out=0,in=0] (C2);
 \draw(h) to[out=180,in=90] (C1) ;   
\end{tikzpicture}
\end{minipage}
$$
\caption{A $C_2$-move and a $C_4$-move.}\label{fig:nmove}
\end{figure}

The basic tool we need is the following theorem, which follows directly from work of Habiro\cite{Habiro2000}.

\begin{theorem}\label{thm:hab1}
Suppose that $\nu$ is an additive invariant of (unframed) knots. If it is invariant under $C_n$-moves then it is a finite-type invariant of degree $n-1$.
\end{theorem}
\begin{proof}
The $C_n$ moves constructed here are an alternate formulation of clasper surgery, and in fact are very close to Habiro's original formulation in his master's thesis. If we had allowed $E_2$ to replace an arbitrary arc in iterating the construction (as opposed to the upper left), we would get an arbitrary capped clasper surgery. The ones constructed here correspond to ``linear claspers," where all nodes are directly adjacent to a leaf. The topological IHX relation allows us to reduce to this case.  (See Theorem 13 of \cite{CT1}.) So the $C_n$-moves introduced here are equivalent to the $C_n$-moves in \cite{Habiro2000}. The argument in Theorem~\ref{n-moves} works just as well if we allow $E_2$ to replace an arbitrary subarc at each stage, so one could dispense with this subtlety.

The monoid of knots modulo the equivalence relation of $n$-equivalence is a finitely generated abelian group \cite{Goussarov,Habiro2000}. Theorem 6.17 of \cite{Habiro2000} states that the natural projection $\psi_{n-1}$ from knots to this abelian group is a universal additive finite-type invariant of degree $n-1.$ So if $\nu$ is an additive invariant of knots which also is invariant under $C_n$ moves, it induces a homomorphism on the group of knots modulo $n$-equivalence.  It
thus factors as a composition of a degree $n-1$ invariant with a group homomorphism. It is therefore a degree $n-1$ invariant itself.
\end{proof}

In order to move to the framed case, we need the following lemma.
 For any integer $k$, let $\operatorname{fr}_k$ be the map from unframed knots to framed knots which adds a $k$-framing.
\begin{lemma}\label{lem:addframe}
Let $U_1$ represent the $+1$ framed unknot. Let $\nu$ be an additive framed knot invariant taking values in an abelian group. Then $$\nu\circ \operatorname{fr}_k=k\nu(U_1)+\nu\circ\operatorname{fr}_0.$$
\end{lemma}
\begin{proof}
One can push the twisting of the framing onto a standard subarc of the knot to see that  $\operatorname{fr}_k(K)=U_k\# \operatorname{fr}_0(K)$, where $U_k$ is a $k$ framed unknot. Then one separates each of the twists, and uses the fact that $\nu$ is additive. 
\end{proof}
\begin{corollary}
Suppose that $\nu$ is an additive invariant of framed knots. If it is invariant under $C_n$-moves then it is a finite-type invariant of degree $n-1$ for $n\geq 2$.
\end{corollary}
\begin{proof}
Note that $\nu\circ\operatorname{fr}_0$ is an additive invariant of unframed knots, and that it is invariant under $C_n$ moves, since $\nu$ is. Therefore $\nu\circ\operatorname{fr}_0$ is finite-type of type $n-1$ by Theorem~\ref{thm:hab1}. On the other hand, by Lemma~\ref{lem:addframe}, $\nu(K)=\operatorname{fr}(K)\nu(U_1)+\nu\circ\operatorname{fr}_0(K),$ where $\operatorname{fr}(K)$ is the framing number. The invariant $\operatorname{fr}$ is known to be type $1$, so we have a linear combination of a type $n-1$ and type $1$ invariant, which is therefore of type $n-1$.
\end{proof}


%

This corollary gives us the main tool we need to show that $\pi_0(ev_n)$ is of finite type.  We also use deformations of evaluation maps of knots.

\begin{definition}
A \emph{$\sigma$-deformation} of (the evaluation map of)
a knot $K$ is a map obtained by precomposing the adjoint of the evaluation map $\Delta^n \times \Emb^{fr}(\R, \R^3) \to 
C_n\la \In^3, \d \ra \x O(3)^n$ with a section $\sigma : \Delta^n \to \Delta^n \times \Emb^{fr}(\R, \R^3)$ of this trivial bundle which maps into the component of $K$ in the embedding space.
\end{definition}

\begin{theorem}\label{n-moves}
The map $\pi_0(ev_n)$ is invariant under $C_n$-moves, and therefore is a type $n-1$ invariant.
\end{theorem}

\begin{proof}
Let $K$ be a framed knot, and let $K'$ be the knot after the $C_n$-move has been applied. Our strategy is to find $\sigma$-deformations of
$K$ and $K'$ which we can then show are homotopic.   These $\sigma$-deformations 
will have the property that none of the configuration points in $ev_n(\sigma(\vec{t}))$ meet the central disk 
whenever stricly less than $n$ of those configuration points in $ev_n(\sigma(\vec{t}))$ lie on the subarcs of $E_n$. 
This implies that the $\sigma$-deformations are homotopic, since we can just push the arc across the central disk of $E_n$ without ever introducing collisions of configuration points. This is by design for points in $\Delta^n$ for which $n - 1$ points or less are inside $E_n$.  If there are $n$ points in $E_n$, 
that means there is no point left over to lie in the exterior arc that we are homotoping.

Both of our $\sigma$-deformations will isotop only the copies of $E_n$ in $K$ and $K'$.  
Consider a nested copy of $E_i$ inside $E_n$. It consists of an arc $\alpha_i$ clasping with a copy of $E_{i-1}$.
Let $D_i$ be the (pair of) space(s) given by the intersection of $E_i$ with a small neighborhood $D^2 \x (-\delta, \delta)$ of the central disk $D^2 \x \{0\}$, where $\delta$ is large enough so that there is a homeomorphism of pairs $E_i \cong D_i$.  
The isotopies of $E_i$ will slide $D_i$ along $E_i$, i.e., they will take $D_i$ into $D^2 \x (a,b)$ for some $(a,b)\subset [-1,1]$.  Roughly, $(a,b)$ will be a small interval near either $-1$ or $1$ according as which of these endpoints is closer to the center of mass of the configuration points.
More specifically, a set of configuration points in the arc $\alpha_i$ will pull $D_i$ toward that side of $D^2 \x [-1,1]$ in a manner that increases 
as the minimum distance of these points to the midpoint of the arc $\alpha_i$ decreases.  Configuration points in the copy of $E_{i-1}$ exert a similar tug 
 to their end of $E_i$ in a manner which increases as they get closer to the midpoints of their arcs. However the tug of a point is halved in magnitude when you pass to $E_{i-1}$. This has been set up so that a point at the midpoint of $\alpha_i$ will always exert a tug that equals or exceeds the collective tug of $E_{i-1}$.

We also set things up so that if $E_{i-1}$ has fewer than $i-1$ configuration points on it, then configuration points in $\alpha_i$ can never get more than $\eps$ away from their end disk $D^2 \x 0$. The point is that any configuration points in $\alpha_i$ will have a greater tug on the clasp than $E_{i-1}$ when one is at the midpoint of $\alpha_i$. We just ensure that this tug is strong enough to pull $\alpha_i$ $\eps$-close to its end disk. (We can come up with a uniform $\eps$ in this way since there is a discrete gap between the maximal tug that $E_{i-1}$ can exert and the tug it exerts at less than full occupancy.)

Similarly arrange that if the arc $\alpha_i$ has no configuration points, then no configuration points in $E_{i-1}$ can get more than $\eps$ away from $D^2 \x 1$ inside $E_i$.

Ultimately, we have an isotopy of $E_n$ parametrized by $C_n\la \In, \d\ra$, or equivalently a reembedding of $E_n$ for each $(t_1,...,t_n) \in C_n\la \In, \d\ra$.  See Figure \ref{C4MoveExample} for an example.

\begin{figure}[h!]
\includegraphics[scale=0.6]{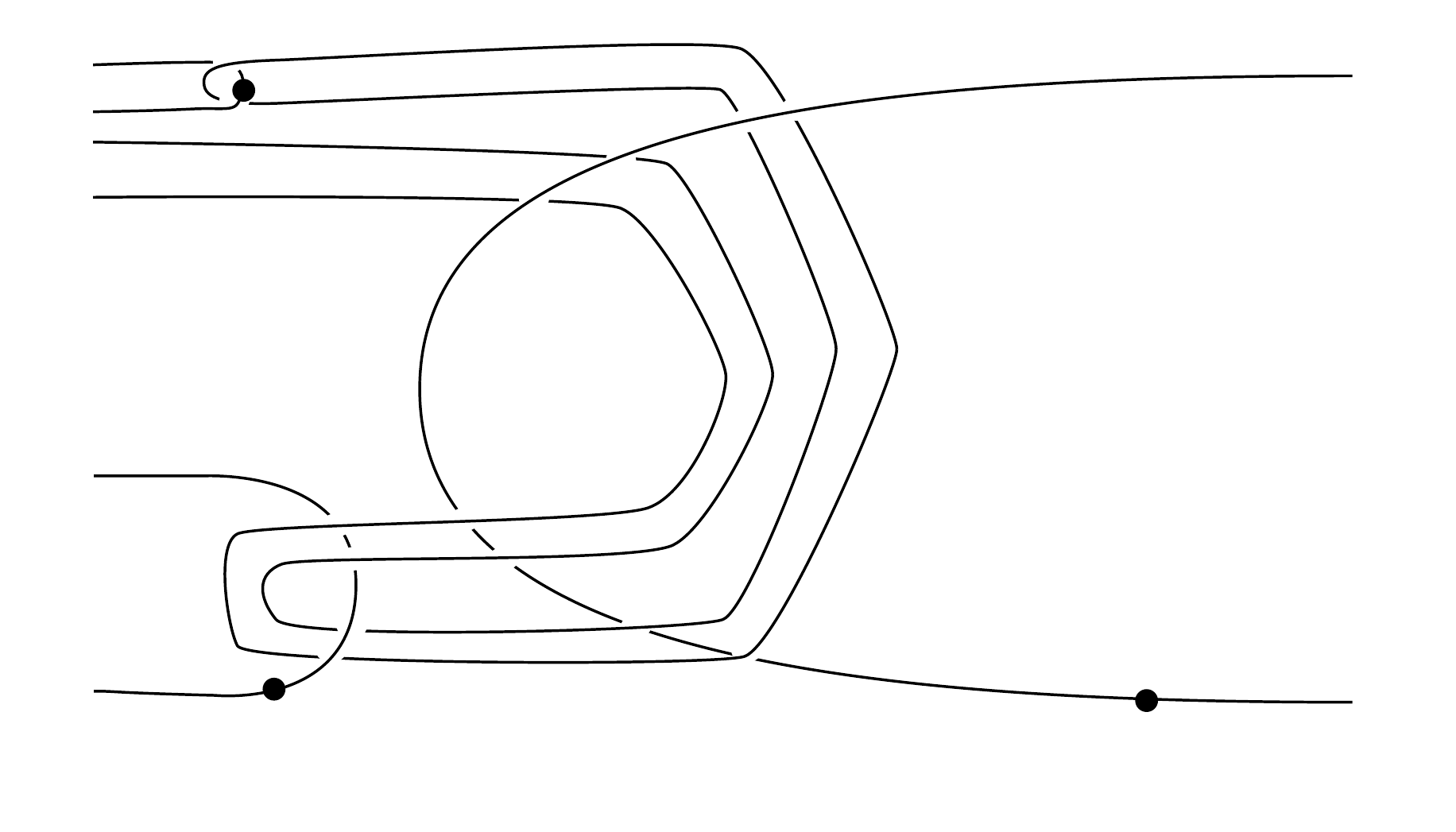}
\caption{An example of the reembedding of $E_4$ resulting from a point $(t_1,...,t_4) \in C_4\la \In, \d\ra$, whose image under $ev_4$ is also shown.  (The image of one of the $t_i$ is not in $E_4$.)
The two left-hand configuration points drag their respective arcs toward their ends of their cylinders.  The right-hand configuration point is partially balancing their tug on the central clasp.  If the right-hand point moved toward the middle of its arc, the central clasp of $E_4$ would get tugged close to the right-hand end of the cylinder, although the other configuration points would still be close to the left-hand end.  A point added to the arc without a point on it would 
bring $E_4$ closer to the original embedding
as it got closer to the midpoint of the arc.}
\label{C4MoveExample}
\end{figure}

With such a family of reembeddings, we claim that no configuration point ever passes through the $D^2 \x 1/2$ disk of $E_n$, provided that strictly less than $n$ configuration points are present inside $E_n$. Let the arcs in $E_n$ be called $\alpha_1$, $\alpha_2$, ... arranged in order of 
decreasing depth.

We know that there is some arc $\alpha_i$ which is not occupied by a configuration point.  The points in $E_{i-1}$ stay within $\eps$ of the end disk and cannot cross the center disk of $E_n$. Thus any points that do cross the center must lie on one of the arcs 
$\alpha_{i+1}$, ..., $\alpha_{n}$. 
Now consider such an arc $\alpha_j$, 
$j > i$.
It links with $E_{j-1}$ which has less than full occupancy.  Points in $\alpha_j$ cannot get further than $\eps$ away from their end disk, and cannot cross the center disk of $E_n$, so we are done!
\end{proof}

\section{The homotopy spectral sequence for the tower and finite-type knot theory}
\label{SpectralSequence}

In this section we further develop the spectral sequence for the homotopy groups and in particular the components
of the stages in the Goodwillie--Weiss tower for classical knots and its relationship - both established and conjectural - with finite-type knot 
theory.  Such analysis for knots in higher-dimensional Euclidean space,
which are connected,  has  been covered elsewhere, starting in \cite{SinhaTop}.
We see that at the $E^2$ stage the entries of the spectral sequence are
 exactly what one would expect if the tower is to serve as a universal additive finite-type invariant.
We in particular see similar structures to what Goodwillie and Weiss \cite[Section 5]{GW} 
originally saw in higher dimensions, but 
can also compare that to newer results on the combinatorics of finite-type invariants \cite{Conant}.

A priori the spectral sequence of a totalization tower, or any other tower of fibrations, is difficult to discern in degree zero.  Not only
is $\pi_0$ only a set-valued functor, but homotopy groups can differ over different components.  We saw in 
Section~\ref{PropertiesOfProjMaps} however that this tower has additional algebraic structure, which leads to the following.

\begin{theorem}
\label{SSofAbGps}
The spectral sequence for the homotopy groups, and in particular components, of $AM^{fr}_n$ as a stage in the Goodwillie--Weiss 
tower is a spectral sequence of abelian groups which converges.
\end{theorem}

\begin{proof}
By Corollary~\ref{C1Fibration} and  Theorem~\ref{Pi0IsGroup} the fiber sequence
\begin{equation}\label{layersequence}
L_n \to AM^{fr}_n \to AM^{fr}_{n-1}
\end{equation}
is a fibration sequence of group-like $\C_1$-spaces.  It is thus loops on the fibration sequence defined on their classifying spaces.
Its long exact sequence in homotopy groups is then a degree shift of that for the classifying spaces, starting with $\pi_1$ of the classifying spaces being $\pi_0$ of these spaces.  These exact sequences can be spliced in the usual way to obtain a spectral sequence, which by Corollary~\ref{Pi0IsGroup} is one of abelian groups.

Convergence follows from the fact that this tower is finite when we truncate it at 
$AM^{fr}_n$, and the groups are finitely generated by Proposition~\ref{E1Term} below along with the fact that homotopy groups of spheres are finitely generated.
\end{proof}

We can now analyze the spectral sequence in further detail.  By Proposition \ref{CubicalFibers}, 
$L_n$ is equivalent to $\Omega^n$ of the total fiber of the $n$-cube $C^{fr}_\bullet \langle \In^3, \d \rangle \circ {\mathcal{G}}^{!}_{n}$.  
Unraveling definitions, this is an $n$-cube whose entries are spaces of configurations of at most $n$ points, together with at most $n$ frames, where each map in the cube forgets a point and a frame.  We can express this cube as an entry-wise product of the cubes $S \mapsto C_{\underline{n} \setminus S} \la \In^3, \d\ra$ and $S \mapsto O(3)^{\underline{n} \setminus S}$.  

The total fiber of the product is the product of the total fibers, and for $n\geq 2$, the total fiber of the cube of $O(3)$'s is a point.  Thus it suffices to consider the total fiber of $S \mapsto C_{\underline{n} \setminus S} \la \In^3, \d\ra$.
Furthermore, we can switch to open configuration spaces, for which these forgetting maps are well known to be fibrations.  We take fibers in one direction and consider the resulting $(n-1)$-cube instead.  
Here we see the entries as $\In^3$ with a finite set of points removed, and maps which are inclusions.
Up to homotopy, this is a cube of spaces $\bigvee_T S^2$ indexed by subsets $T \subset \underline{n-1}$, 
where each map projects off a wedge factor.  Call this cube $\mathcal{P}(\underline{n-1}) \left( \bigvee S^2 \right)$.

By Hilton's Theorem \cite{Hilton}, the homotopy groups of a wedge $\bigvee_{n} S^2$ is a direct sum of homotopy groups of higher-dimensional spheres.  To elaborate, let $\mathcal{L}_{n}$ be the free graded Lie algebra (working over the integers for the rest of this section) on $n$ odd-graded generators in degree one. Let $\mathcal{B}_{n, full}$ be a basis for $\mathcal{L}_{n}$, choosing these consistently as $n$ varies.  For example, we could use a graded version of Hall bases.  

 Hilton's Theorem states that $\pi_* (\bigvee_{n} S^2)$ is a direct sum $\bigoplus_{W \in \mathcal{B}_{n,full}} \pi_* S^{|W| - 1}$, where $W$ is the degree or word length of $W$.
The theorem is functorial if we use bases for free Lie algebras of different ranks
which extend one another, since the Whitehead products used to define the elements of homotopy are functorial. Because the
projection maps between wedge products of $S^2$ are split,  these different bases split off.  An immediate inductive calculation 
of the homotopy groups of the total fiber (as an iterated fiber of fibers) shows that what is left for homotopy groups
is indexed by a basis $\mathcal{B}_n$ of the submodule of the
$\mathcal{L}_{n}$ spanned by brackets in which all generators occur.

\begin{proposition}\label{E1Term}
The spectral sequence for the homotopy groups (including $\pi_0$) of $AM_n$ has as $E^1_{-p, *}$ the module 
$\bigoplus_{W \in \mathcal{B}_{p-1}} \pi_* S^{|W| - 1}$.
\qed
\end{proposition}

By Lemma \ref{LoopspaceStructuresAgree}, the abelian group structure in the spectral sequence 
agrees with the usual abelian group structure on the homotopy groups of spheres.

We now focus on total degree zero.  
Let $\mathcal{A}^I_n$ be the $\Z$-module of chord diagrams on a line segment with $n$ chords, modulo the usual four-term relation 4T and the relation SEP, which sets every separated (i.e. non-primitive) diagram to zero.  Alternatively, $\mathcal{A}^I_n$ is the $\Z$-module of trivalent diagrams, modulo antisymmetry, the IHX relation, and SEP.  See \cite{Conant} for more details.

\begin{theorem}\label{E2Comp}
The group $E^2_{-n,n}$ is isomorphic to $\mathcal{A}^I_{n-1}$.
\end{theorem}

\begin{proof}
 By Proposition~\ref{E1Term} the group $E^1_{-n,n}$  is isomorphic to the submodule 
 of the free Lie algebra on $n-1$ generators generated by $(n-1)$-fold brackets where each generator appears exactly once.  
 This module is $\mathcal{L}ie(n-1)$, the $(n-1)$-st space in the Lie operad, which well known to be free of rank $(n-1)!$.  
  
Next, we consider the 1-line of the $E^1$ page.  Under the identification of Proposition~\ref{E1Term} these groups decompose
into a free summand and two-torsion.  The free summands are indexed by $n$-fold brackets in the free Lie algebra on $n-1$ generators,
again in which all generators occur.  The two-torsion summands occur as composites of $S^{n} \overset{\eta}{\to} S^{n-1}$ with $(n-1)$-fold 
Whitehead products from $S^{n-1} \to \mathcal{P}(\underline{n-1}) \left( \bigvee S^2 \right)$.  This summand is this isomorphic to 
$\mathcal{L}ie(n-1) \otimes \Z/2$.

The differential $d_1$ must be zero on the torsion summand.  We claim that on the free summand the
differential is the integral version of the differential defined in 
\cite{Scannell-Sinha}.  There in Theorem 2.1, through the tower of fibrations 
\begin{equation}\label{towerforconf}
C_n (\R^d) \to C_{n-1} (\R^d) \to \cdots \to C_0 (\R^d)
\end{equation}
the well-known rational homotopy Lie algebra of the configuration space $C_n (\R^d)$ is calculated as generated by classes $b_{ij}$ in degree $d-1$.
Under the map from the total fiber of $\mathcal{P}(\underline{n-1}) \left( \bigvee S^2 \right)$ to $C_n (\R^3)$ the basis for the
free Lie algebra on generators, say $x_i$, sends a bracket to a corresponding brackets in the $b_{in}$'s.    Because
the projections in the tower  (\ref{towerforconf}) are split, these brackets are integral generators of free summands as well.
(In fact, one can use the splitting of the tower  (\ref{towerforconf}) and the Hilton--Milnor theorem to express homotopy groups of configuration spaces
as a direct sum of homotopy groups of spheres.)
The formulas for the differential given in \cite{Scannell-Sinha} are given in terms of these integral generators, so they
 hold for the spectral sequence over the integers as well.

In \cite{Conant} the cokernel of the rational $d^1$ is computed to be $\mathcal{A}^I_{n-1} \otimes \Q$.  
While the result is stated rationally (which
is where the conjecture was made), all of the calculations involve only integer coefficients.
\end{proof}

\bigskip

At the $E^2$-level the components of $AM_n$ thus look like they should receive a universal  additive finite 
type-$(n-1)$ invariant over the integers, which was established for $n=3$ as the main result of \cite{BCSS}.  

We elaborate our conjecture as follows.  Because the map $ev_n$ from the knot space to this 
Goodwillie--Weiss tower, sometimes called the ``homotopy tower'', factors the map to the variant that Volic 
considers in \cite{VolicFT}, sometimes called the ``rational homology tower'', 
we already knew that it encodes every rational finite type-$\frac{n}{2}$ invariant.
By Proposition~\ref{n-moves}, we now know that any invariant which factors through $\pi_0(ev_n)$ is of type at most $(n-1)$.
Standard finite-type theory then gives a map from $\mathcal{A}^I_{n-1}$ to $\pi_0(AM_n)$, namely 
by taking alternating sums of values on resolutions of knots with singularities described by a given chord diagram.    
We conjecture that this is an isomorphism at $E^2$, which collapses to $E^\infty$.  This would imply by 
Theorem~\ref{E2Comp} that all weight systems lift to finite-type invariants over the integers.  That is, it would establish 
$\pi_0(ev_\infty)$ as a refinement of the Kontsevich integral, defined over the integers.\\

In the framed setting we have some additional low-dimensional calculations. Namely, 
$AM_1^{fr} \simeq \Omega SO(3)$, implying its components are isomorphic to $\Z/2$.  The evaluation map calculates the parity of the 
framing, which is indeed the only additive type-one invariant.   Next,  Theorem~3.6 of \cite{BCSS} states that $AM_2$ is contractible.  
The subcubical diagram which defines $AM^{fr}_2$ fibers over that which defines $AM_2$, with the maps in this fibering built 
from the standard fibration of $SO(3)$ over $S^2$ as the 
unit tangent bundle.  The fiber is a subcubical model for $\Omega S^1 \simeq \Z$ 
(pulled back from the cosimplicial model for $\Omega S^1$ via $\mathcal{G}_2$), and the map from the framed knot space 
to $AM_2^{fr}$ classifies framing number.  

More generally,  the subcubical diagram which defines $AM^{fr}_n$ fibers over that which defines $AM_n$ with fiber given by a $n$-subcubical diagram
which models $\Omega S^1$.  Because $\pi_0 AM^{fr}_n$ projects surjectively onto $\pi_0 AM^{fr}_2$, compatibly with the identification
of this fiber with $\Omega S^1$, on components this yields a splitting 
$\pi_0 AM^{fr}_n \cong \pi_0 AM_n \times \Z$.  (Note that this would follow immediately from  elementary results in Goodwillie-Weiss calculus
if the splitting of spaces of framed knots can be made functorial.)
Thus these Goodwillie--Weiss models reflect the fact that the usual and framed finite-type theories
differ only by the framing number invariant.  

One key step towards establishing this conjecture would be the collapse of the spectral sequence, which is now of a tower of fibrations amenable to tools from algebraic topology.  This is in contrast to Vassiliev's approach, where the limiting process of unstable spectral sequences is  not well understood.  
(See \cite{Giusti} for this limiting process in a piecewise-linear setting.) The Goodwillie-Weiss tower is built from maps of spaces, in particular the sequences of (\ref{layersequence}) and the diagrams which define $AM_n$, 
 which  have been analyzed to great effect for knots in higher dimensions \cite{LTV}.

Another key intermediate step would be the surjectivity on components of $ev_n$.  This statement would follow from deep connectivity results of Goodwillie--Klein \cite{GK} if those applied in codimension two, but it may be approachable more directly in this case.

We suspect, however, that direct analysis of the invariants which arise from $\pi_0 ev_n$ will be most fruitful.  They have already led to new geometric insight in degrees three and four \cite{BCSS, Flowers}.  Sinha--Walter's Hopf invariants \cite{SinhaWalter13} can fully be applied by the calculations of Proposition~\ref{E1Term}.  They seem to lead to 
Goussarov--Polyak--Viro formulae, which is a promising sign.

\bibliographystyle{alpha}
\bibliography{refs}

\end{document}